\documentclass[arxiv,reqno,twoside,a4paper,12pt]{amsart}

\usepackage{amsmath, verbatim}
\usepackage{amssymb,amsfonts,mathrsfs,mathtools}
\usepackage[colorlinks=true,linkcolor=blue,urlcolor=blue,citecolor=blue]{hyperref}
\usepackage[driver=pdftex,heightrounded=true,centering]{geometry}
\usepackage{longtable, slashed}
\usepackage{tabularx}
\usepackage{multirow}
\usepackage{caption}
\usepackage{latexsym,enumitem}
\usepackage[all]{xy}
\usepackage[latin2]{inputenc}
\usepackage{t1enc, indentfirst}
\usepackage{graphicx, graphics}
\usepackage{pict2e, epic, amssymb }
\usepackage{tikz, pst-node}
\usepackage{tikz-cd, pgfplots}
\usetikzlibrary{arrows}
\usetikzlibrary{calc, patterns}
\pgfplotsset{compat=1.9}
\usepackage{epstopdf}
\usepackage{float}

\usepackage[toc]{appendix}
\usepackage[mathbf]{euler}

\usepackage[UKenglish]{babel}
\usepackage{tikz}
\usepackage{tikz-3dplot}

\usepackage{diagmac2} 
\usepackage{array} 
\usetikzlibrary{positioning}
\usetikzlibrary{calc}
\usetikzlibrary{through}
\usepackage{subcaption}
\allowdisplaybreaks 

\usepackage{xcolor}

\usepackage[scaled]{helvet} 
\usepackage{courier}

\theoremstyle{break}
\newtheorem{thm}{Theorem}[section]
\newtheorem{lem}[thm]{Lemma}%
\newtheorem{prop}[thm]{Proposition}
\newtheorem{cor}[thm]{Corollary}
\newtheorem{defn}[thm]{Definition}
\newtheorem{rmk}[thm]{Remark} 

\DeclareMathOperator{\Id}{id}

\DeclareMathOperator{\dvol}{dvol}
\DeclareMathOperator{\di}{d}
\DeclareMathOperator{\supp}{supp}
\DeclareMathOperator{\op}{op}
\DeclareMathOperator{\ff}{ff}
\DeclareMathOperator{\fd}{fd}
\DeclareMathOperator{\lf}{lf}
\DeclareMathOperator{\rf}{rf}
\DeclareMathOperator{\td}{td}

\DeclareMathOperator{\vol}{vol}
\DeclareMathOperator{\diff}{Diff}
\DeclareMathOperator{\diam}{diam}
\DeclareMathOperator{\pt}{pt}

\numberwithin{equation}{section}

\definecolor{qqwuqq}{rgb}{0,0,0}

\begin{document}

\title[Parametrix construction]{Heat-type Equations on manifolds \\ with fibered boundaries II: \\ Parametrix construction}

\author{Bruno Caldeira}
\address{Universidade Federal de S\~{a}o Carlos, Brazil}
\email{brunoccarlotti@gmail.com}
\email{brunocarlotti@estudante.ufscar.br}

\author{Giuseppe Gentile}
\address{Leibniz Universit\"{a}t Hannover, Germany}
\email{giuseppe.gentile@math.uni-hannover.de}


\maketitle
\begin{abstract} 
This is the second part of a two parts work on the analysis of heat-type equations on manifolds with fibered boundary equipped with a $\Phi$-metric.
This setting generalizes the asymptotically conical (scattering) spaces and includes special cases of magnetic and gravitational
monopoles.
The core of this second part consists on the construction of parametrix for heat-type equations.
Consequently we use the constructed parametrix to infer results regarding existence and regularity of certain homogeneous and non homogeneous second order linear parabolic equations with non constant coefficients.
This work represents the first step towards the analysis of geometric flows such as Ricci-, Yamabe and Mean Curvature flow on some families of non compact manifolds. 
\end{abstract}

\tableofcontents

\section{Introduction and statement of the main results}\label{Introduction}

In the first part (\cite{paper1}) of this two parts work the authors presented mapping properties for the heat-kernel operator $\mathbf{H}$ and derived existence and uniqueness of the heat equation on a $\Phi$-manifold.
The aim of the present work is to extend the analysis carried over in \cite{paper1} to a slightly more general family of equations.
Namely we consider some linear parabolic equations with variable coefficients on $\Phi$-manifolds which we refer to as heat-type equations.

\medskip
Manifolds with fibered boundary are a class of compact manifold $\overline{M}$ whose boundary $\partial \overline{M}$ is the total space of a fibration $\phi:\partial \overline{M}\rightarrow Y$ over a closed (i.e. compact without boundary) Riemannian manifold $Y$. 
Moreover, the fibers of the fibration $\phi$ are copies of a fixed closed Riemannian manifold $Z$.
An open manifold $M$, which is the interior of a manifold fibered boundary $\overline{M}$, is a $\Phi$-manifold if it is equipped with a specific Riemannian metric known as $\Phi$-metric.
Such a metric is such that, near the boundary $\partial\overline{M}$, has asymptotic behavior described by 
\begin{equation}\label{FirstEquation}
g_{\Phi} = \dfrac{\di x^{2}}{x^{4}} + \dfrac{\phi^{*}g_{Y}}{x^{2}} + g_{Z} + h,
\end{equation}
where $h$ is the collection of cross-terms and it contains extra powers of $x$ in each of its terms.  
In the above, $g_{Y}$ is a Riemannian metric on the base $Y$, while $g_{Z}$ is a symmetric bilinear form on $\partial \overline{M}$ which restricts to a Riemannian metric at each fiber. 
\medskip

The simplest example of a $\Phi$-manifold is $\mathbb{R}^m$ equipped with the Euclidean metric expressed in polar coordinates
$$g=\di r^2+r^2\di \theta.$$
In fact, to obtain an expression as the one in \eqref{FirstEquation} from the above, one could simply perform a change of coordinates $x=r^{-1}$ far from the origin. 
In this case, note that $Y = \mathbb{S}^{m-1}$ and $Z=\{\pt\} $.  
Other example of $\Phi$-manifolds include several complete Ricci-flat metrics, products of locally Euclidean spaces with a compact manifold and some classes of gravitational instantons.
\medskip

Despite the fact that $\Phi$-manifold have been firstly introduced in 1990's, they remain relatively new in the field of Geometric Analysis and, in particular, in the analysis of geometric flows such as Yamabe-,Ricci- and the Mean Curvature flow, among others.
This paper can be thought as a preparation for the analysis of the above mentioned flows. 
Indeed we prove short-time existence for Cauchy problems of the form
\begin{equation}\label{baseq}
(\partial_{t} + a\Delta)u = \ell, \;\;u|_{t=0}=u_0,
\end{equation}
for some suitable functions $\ell$ and $a$ and $u_0$.
It is well known that (most) geometric flows give rise to quasilinear parabolic PDE's but the arguments treated here can be tweaked a bit (e.g. by linearizing the quasi-linear equation) to guarantee short-time existence for such geometric flows, as it has been done by the first named author for the Yamabe flow in \cite{bruno} and by the second named author for the mean curvature flow in \cite{giu}.

\subsection{Main results and structure of the paper}

Our aim is to extend the results in \cite{paper1} to Cauchy problems of the form \eqref{baseq}.
This is achieved by making use of the mapping properties proved by the authors in \cite{paper1}.
Therefore in \S \ref{review-sec} we give an overview on $\Phi$-manifolds and their properties.
Moreover we recall the definition of the "geometry adapted" H\"{o}lder spaces and the mapping properties of the heat-kernel between these H\"{o}lder spaces.
\S \ref{maxsection} is devoted to the discussion of a parabolic maximum principle, based on the Omori-Yau maximum principle for stochastically complete manifolds.
\begin{thm} \label{max-princ-thm}
Let $(M,g)$ be a stochastically complete manifold and let $a$ be a function on $M$ which is bounded and bounded from below away from zero.  If $u \in C^{2,\alpha}(M\times [0,T])$ is a solution of the Cauchy problem
\begin{equation}
    (\partial_{t} + a\Delta)u = 0;\;\; u|_{t=0} = 0,
\end{equation}
then $u = 0$.
\end{thm}

Based on \cite{bahuaud2019long} we employ the maximum principle in \ref{max-princ-thm} to construct a parametrix for heat-type operators in \S \ref{ParametrixSection}.
In particular we prove:
\begin{thm} \label{theorem2}
	Consider a function $a \in C^{\beta}_{\Phi}(M\times [0,T])$ positive and bounded away from zero.
Then for any $\alpha < \beta$ and for any $\gamma\in\mathbb{R}$ there exist two bounded operators 
	\begin{align*}
	&\mathbf{Q}:x^{\gamma}C^{\alpha}_{\Phi}(M\times [0,T])\rightarrow x^{\gamma}C^{2,\alpha}_{\Phi}(M\times [0,T]), \\
	&\mathbf{E}:x^{\gamma}C^{\alpha}_{\Phi}(M)\rightarrow x^{\gamma}C^{2,\alpha}_{\Phi}(M\times [0,T]),
	\end{align*}
so that the homogeneous and inhomogeneous Cauchy problems
	\begin{align}
	&(\partial_{t} + a\Delta)u = \ell; \;\; u|_{t=0} = 0, \\
	&(\partial_{t} + a\Delta)u = 0; \;\; u|_{t=0} = u_0
	\end{align}
	have solutions $\mathbf{Q}\ell$ and $\mathbf{E}u_{0}$ respectively.
\end{thm}
Finally, in \S \ref{ShortTimeExistenceSec} we generalize the short-time existence and regularity result previously obtained by the authors in \cite{paper1}.
In particular we prove short-time existence and regularity of solutions to a class of linear parabolic equation with variable coefficients.
\begin{cor} \label{theorem4}
Let $\alpha,\beta\in (0,1)$ with $\alpha<\beta$.
Consider the Cauchy problem 
\begin{equation} \label{shorttime}
(\partial_{t} + a\Delta)u = F(u), \;\;
u|_{t=0} = 0,
\end{equation}
with coefficient $a \in C^{\beta}_{\Phi}(M\times [0,T])$ positive and bounded from below away from zero.
Furthermore, assume the map
$F:x^{\gamma}C^{2,\alpha}_{\Phi}(M\times [0,T])\rightarrow C^{\alpha}_{\Phi}(M\times [0,T])$ to satisfy the following conditions: one can write $F = F_{1} + F_{2}$, with 
	\begin{enumerate}
	    \item $F_{1}:x^{\gamma}C^{2,\alpha}_{\Phi}\rightarrow x^{\gamma}C^{1,\alpha}_{\Phi}(M\times [0,T]),$
	    \item $F_{2}:x^{\gamma}C^{2,\alpha}_{\Phi}\rightarrow x^{\gamma}C^{\alpha}_{\Phi}(M\times [0,T])$
	\end{enumerate}
	and, for $u,u' \in x^{\gamma}C^{2,\alpha}_{\Phi}(M\times [0,T])$ satisfying $\|u\|_{2,\alpha,\gamma},\|u'\|_{2,\alpha,\gamma} \le \mu$, exists some $C_{\mu}>0$ such that
	
	\begin{enumerate}
		\item $\|F_{1}(u) - F_{1}(u')\|_{1,\alpha,\gamma} \le C_{\mu}\|u-u'\|_{2,\alpha,\gamma}$, $\|F_{1}(u)\|_{1,\alpha,\gamma} \le C_{\mu},$
		\item $\|F_{2}(u) - F_{2}(u')\|_{\alpha,\gamma} \le C_{\mu}\max\{\|u\|_{2,\alpha,\gamma},\|u'\|_{2,\alpha,\gamma}\}\|u-u'\|_{2,\alpha,\gamma}$, \newline $\|F_{2}(u)\|_{\alpha,\gamma} \le C_{\mu}\|u\|^{2}_{2,\alpha,\gamma}.$
	\end{enumerate}
	Then there exists a unique $u^{*} \in x^{\gamma}C^{2,\alpha}_{\Phi}(M\times [0,T'])$ solution for (\ref{shorttime}) for some $T'>0$ sufficiently small.
\end{cor}

%

\subsubsection*{Acknowledgements} The authors wish to thank Boris Vertman for the supervision as advisor for their Ph.D. theses.  The authors wish to thank the University of Oldenburg for the financial support and hospitality.  
The first author wishes also to thank the Coordena\c{c}\~{a}o de Aperfei\c{c}oamento de Pessoal de N\'{i}vel Superior (CAPES-Brasil- Finance Code 001) for the financial support (Process 88881.199666/2018-01).

\section{Review of part I}\label{review-sec}

As mentioned in \S \ref{Introduction}, this section is dedicated to recollect the main points of \cite{paper1}.  

\subsection{Geometry of $\Phi$-manifolds}\label{PhiMfldsSect}

We say that a compact manifold with boundary $\overline{M}$ has fibered boundary if its boundary $\partial \overline{M}$ is the total space of a fibration
\begin{equation}
    Z \hookrightarrow \partial \overline{M} \xrightarrow{\phi} Y,
\end{equation}
where both $Y$ and $Z$ are closed manifolds of dimensions $b$ and $f$ respectively.  Moreover, consider $g_{Y}$ a Riemannian metric on $Y$ and $g_{Z}$ a symmetric bilinear form on $\partial \overline{M}$ which restricts to Riemannian metrics on each fiber.  Assume, furthermore, 
\begin{equation*}
    \phi:(\partial \overline{M}, \phi^{*}g_{Y} + g_{Z})\rightarrow (Y,g_{Y}) 
\end{equation*}
to be a Riemannian submersion.
Finally, we will denote by $x\in C^\infty(\overline{M})$ the total boundary defining function of $\partial \overline{M}$.
That is $\partial \overline{M}=\{x=0\}$ and the differential $\di x$ never vanishes on $\partial\overline{M}$.
\begin{defn}
A $\Phi$-metric on $M$, which is the open interior or $\overline{M}$, is a Riemannian metric $g_{\Phi}$ that, on a collar neighborhood $\mathcal{U} \simeq (0,1)\times \partial \overline{M}$, can be expressed as
\begin{equation}
    g_{\Phi} = \dfrac{\di x^{2}}{x^{4}} + \dfrac{\phi^{*}g_{Y}}{x^{2}} + g_{Z} + h =: \widehat{g} + h,
\end{equation}
\end{defn}
where $|h|_{\widehat{g}}  = O(x)$.  A pair $(M,g_{\Phi})$ is called a $\Phi$-manifold.

\medskip
Note that, due to the fibration assumption at the boundary, $U$ can be covered by open coordinated charts $U_{i}$ on which every point $p \in U_{i}$ can be written as a triple $(x,y,z)$, where $y = (y_{1},\cdots,y_{b})$ and $z = (z_{1},\cdots,z_{f})$ are lifts of base and fiber coordinates, respectively.

Following \cite{mazzeo1998pseudodifferential}, the most reasonable family of vector fields to consider for the analysis are $\Phi$-vector fields.
The Lie algebra of $\Phi$-vector fields is denoted by $\mathcal{V}_{\Phi}(\overline{M})$ and $\Phi$-vector fields are locally spanned by
\begin{equation}
    x^{2}\partial_{x},\; x\partial_{y_{1}},\; \cdots \;, x\partial_{y_{b}}, \partial_{z_{1}},\; \cdots \;, \partial_{z_{f}}.
\end{equation}
\begin{rmk}
Note that $\Phi$-vector fields have bounded norm with respect to the $\Phi$ metric $g_{\Phi}$.
\end{rmk}
One can now recursively define $\Phi$-$k$-differentiable functions as follows:
\begin{equation} \label{ckphi}
\begin{split}
    &C^{1}_{\Phi}(\overline{M}) = \left\{u \in C^{0}(\overline{M})\; | \; Vu \in C^{0}(\overline{M}) \; \mbox{for every} \; V \in \mathcal{V}_{\Phi}(\overline{M})  \right\},\\
    &C^{k}_{\Phi}(\overline{M}) = \left\{u \in C^{k-1}_{\Phi}(\overline{M})\; | \; Vu \in C^{k-1}_{\Phi}(\overline{M}) \; \mbox{for every} \; V \in \mathcal{V}_{\Phi}(\overline{M})\right\},
\end{split}
\end{equation}
where $k \in \mathbb{Z}_{\geq 2}$.  Since $\mathcal{V}_{\Phi}(\overline{M})$ is a Lie algebra and a $C^{\infty}(\overline{M})$ module, we can consider the algebra $\diff^{*}_{\Phi}(\overline{M})$ of $\Phi$-differential operators.  
In particular a $\Phi$-$k$-differential operator $P\in \diff_{\Phi}^k(\overline{M})$  is a map $P:C_{\Phi}^{\infty}(\overline{M})\rightarrow C_{\Phi}^{\infty}(\overline{M})$ so that it can locally be expressed as 
\begin{equation}
    P = \displaystyle\sum_{|\alpha| + |\beta| + q \leq k}P_{\alpha,\beta,q}(x,y,z)\;x^{2q + |\beta|}\;\partial_{x}^{q}\;\partial_{y}^{\beta}\;\partial_{z}^{\alpha},
\end{equation}
where $\alpha$ and $\beta$ are multi-indices, each $P_{\alpha,\beta,q}$ is a smooth function, $\partial_{y} = \partial_{y_{1}}\cdots \partial_{y_{b}}$ and $\partial_{z} = \partial_{z_{1}}\cdots \partial_{z_{f}}$.  For simplicity, we often denote $\diff^{k}_{\Phi}(\overline{M})$ as $\mathcal{V}^{k}_{\Phi}$.

\subsection{Stochastic completeness of $\Phi$-manifolds}\label{stoc-comp-subsect} 

A crucial property of $\Phi$-manifolds, as highlighted in \cite[\S 3]{paper1} is that they are stochastically complete.
In our previous work stochastic completeness has been used to deduce mapping properties of the heat-kernel.
In the current work we will employ stochastic completeness to make use of the Omori-Yau maximum principle.

\medskip
A Riemannian manifold $(M,g)$ is said to be stochastically complete if the heat kernel $H(t,p,\widetilde{p})$ of the (positive) Laplace-Beltrame operator $\Delta$ satifies
\begin{equation}
    \int_{M}H(t,p,\widetilde{p})\dvol_{g}(\widetilde{p}) = 1,
\end{equation}
for every $t\geq 0$ and $p \in M$.

In particular, as shown in \cite[\S 3]{paper1}, $\Phi$-manifolds are stochastically complete because the function
\begin{equation}\label{vol-growth}
    \overline{f}(\cdot) := \dfrac{\cdot}{\log(\vol(B(p,\cdot)))} \notin L^{1}(1,+\infty).
\end{equation}
We remind the reader that (for complete manifolds) condition \eqref{vol-growth} is enough to conclude stochastic completeness as stated in \cite[Theorem 2-11]{alias2016maximum} (see also \cite{grigor1986stochastically}).

\subsection{H\"older continuity on $\Phi$-manifolds}\label{HoldPhiSect}

Next we present H\"{o}lder spaces suitable for our analysis.
As mentioned in the introduction, these spaces are "geometry-adapted" meaning that the distance function as well as the vector fields employed in the definitions encode the singularities arising from the $\Phi$-metric.
More precisely, let $0<\alpha<1$ and $u\in C^0(M\times [0,T])$, for some $T>0$.
We define
\begin{equation}\label{alphanorm}
    \|u\|_{\alpha} = \|u\|_{\infty} + \sup\left\{\dfrac{|u(p,t) - u(p',t')|}{d_{\Phi}(p,p')^{\alpha} + |t-t'|^{\alpha/2}}\right\} =: \|u\|_{\infty} + [u]_{\alpha},
\end{equation}
where the distance function $d_{\Phi}$ between $p = (x,y,z)$ and $p' = (x',y',z')$, is expressed locally as
\begin{equation}
    d_{\Phi}(p,p') = \sqrt{|x-x'|^{2} + (x+x')^{2}\|y-y'\|^{2} + (x+x')^{4}\|z-z'\|^{2}}.
\end{equation}
Thus we define the space of $\alpha$-H\"{o}lder continuous functions by
\begin{equation}\label{Holdernoweights}
    C^{\alpha}_{\Phi}(M\times[0,T]) = \{u \in C^{0}(\overline{M}\times [0,T])\; | \; \|u\|_{\alpha} < +\infty\}.
\end{equation}
As it is natural, we define $\alpha$-H\"{o}lder spaces with higher regularity by
\begin{equation}
    C^{k,\alpha}_{\Phi}(M\times[0,T]) := \left\{u \in C^{0}(\overline{M}\times [0,T])\; \bigg| \; \begin{array}{l}
        \mathcal{V}^{l_{1}}_{\Phi}\partial_{t}^{l_{2}}u \in C^{\alpha}_{\Phi}(M\times [0,T]),\\ 
        with \; l_{1} + 2l_{2} \leq k
    \end{array}\right\},
\end{equation}
where $k \in \mathbb{Z}_{\geq 0}$.  
For each pair $(k,\alpha)$ as above, $C^{k,\alpha}_{\Phi}(M\times [0,T])$ is a Banach space endowed with the norm
\begin{equation}
    \|u\|_{k,\alpha} := \displaystyle\sum_{l_{1}+2l_{2}\leq k}\sum_{V \in \mathcal{V}^{l_{1}}_{\Phi}}\|(V\circ \partial^{l_{2}}_{t})u\|_{\alpha}.
\end{equation}
It follows directly from the definition that $C^{k_{2},\alpha}_{\Phi}(M\times [0,T])\subset C^{k_{1},\alpha}_{\Phi}(M\times [0,T])$ whenever $0\leq k_{1} \leq k_{2}$.

\medskip
We can generalize to weighted-H\"older spaces as follows: for $\gamma \in \mathbb{R}$, define
\begin{equation}
    \begin{split}
        &x^{\gamma}C^{k,\alpha}_{\Phi}(M\times [0,T]) := \{x^{\gamma}u\; | \; u \in C^{k,\alpha}_{\Phi}(M\times [0,T])\}, \\
        &\|x^{\gamma}u\|_{k,\alpha,\gamma} := \|u\|_{k,\alpha}.
    \end{split}
\end{equation}
The pair $(x^{\gamma}C^{k,\alpha}_{\Phi}(M\times [0,T],\|\cdot\|_{k,\alpha,\gamma})$ is a Banach space as well.  One can conclude this simply by noticing that the operator ``multiplication by $x^{\gamma}$'' $\textbf{M}(x^{\gamma})$ is an isometry between $C^{k,\alpha}_{\Phi}$ and $x^{\gamma}C^{k,\alpha}_{\Phi}$.

\subsection{Mapping properties on $\Phi$-manifolds}\label{heatkerelmap-subsect}

The mapping properties of the heat-kernel op $\mathbf{H}$  proved in \cite{paper1} will play a key role in the construction of the parametrix for heat-type operators.
Therefore, for the sake of completeness, we present them here.
We refer the interested reader to our previous work for a very detailed analysis.

\medskip
For a function $u:M\times [0,T]\rightarrow \mathbb{R}$, $T>0$, define the function $\mathbb{H}u$ by convolution with the heat-kernel associated to the unique self-adjoint extension of the positive Laplace-Beltrami operator $\Delta_{\Phi}$.
That is 
\begin{equation}\label{heatkerelop-def}
    \mathbf{H}u(p,t) := \int_{0}^{t}\int_{M}H(t-\widetilde{t},p,\widetilde{p})u(\widetilde{p},\widetilde{t})\dvol_{\Phi}(\widetilde{p})\di \widetilde{t},
\end{equation}
By making use of the asymptotic behavior of the heat-kernel $H$ provided in \cite[Theorem 7.2]{VerTal}, we proved:
\begin{thm}\label{heatkernelmap}\label{MappingPropertiesTHM}\label{SecondMappingPropertyTHM}\label{mappropH}
\cite[Theorem 1.1]{paper1}  Let $(M,g_{\Phi})$ be a $\Phi$-manifold.  Then, for any $0<\alpha<1$, $k \in \mathbb{Z}_{\geq 0}$, $\gamma \in \mathbb{R}$ and $T>0$, the heat-kernel operator acts continuously as follows:
\begin{equation}
    \begin{split}
        &\mathbf{H}:x^{\gamma}C^{k,\alpha}_{\Phi}(M\times [0,T])\rightarrow x^{\gamma}C^{k+2,\alpha}_{\Phi}(M\times [0,T]), \\
        &\mathbf{H}:x^{\gamma}C^{k,\alpha}_{\Phi}(M\times [0,T])\rightarrow \sqrt{t}\;x^{\gamma}C^{k+1,\alpha}_{\Phi}(M\times [0,T]), \\
        &\mathbf{H}:x^{\gamma}C^{k,\alpha}_{\Phi}(M\times [0,T])\rightarrow  t^{\alpha/2}x^{\gamma}C^{2}_{\Phi}(M\times [0,T]).
    \end{split}
\end{equation}
\end{thm}
Consequently, we proved the following result regarding short-time existence and regularity solutions for the heat equation oh $\Phi$-manifolds.

\begin{thm}\label{cor-paper1}
\cite[Corollary 1.2]{paper1} Let $\alpha,k,\gamma$ and $T$ be as in Theorem \ref{heatkernelmap} and consider the nonlinear Cauchy problem
\begin{equation} \label{CP1}
    (\partial_{t} + \Delta_{\Phi})u = F(u),\; u|_{t=0} = 0.
\end{equation}
Assume $F$ to satisfy the following conditions:
\begin{enumerate}
\item $F:x^{\gamma}C^{k+2,\alpha}_{\Phi}(M\times [0,T])\rightarrow C^{k,\alpha}_{\Phi}(M\times [0,T])$;
\item $F$ can be written as a sum $F = F_{1} + F_{2}$ with 
\begin{itemize}
\item[i)]$F_{1}:x^{\gamma}C^{k+2,\alpha}_{\Phi}\rightarrow x^{\gamma}C^{k+1,\alpha}_{\Phi}(M\times [0,T]),$
\item[ii)] $F_{2}:x^{\gamma}C^{k+2,\alpha}_{\Phi}\rightarrow x^{\gamma}C^{k,\alpha}_{\Phi}(M\times [0,T]);$
\end{itemize}
\item For $u,u' \in x^{\gamma}C^{k+2,\alpha}_{\Phi}(M\times [0,T])$ with $\|\cdot\|_{k+2,\alpha,\gamma}$-norm bounded from above by some $\eta>0$, i.e. $\|u\|_{k+2,\alpha,\gamma},\|u'\|_{k+2,\alpha,\gamma} \le \eta$,
there exists some $C_{\eta}>0$ such that
\begin{itemize}
\item[i)] $\|F_{1}(u) - F_{1}(u')\|_{k+1,\alpha,\gamma} \le C_{\eta}\|u-u'\|_{k+2,\alpha,\gamma}$, $\|F_{1}(u)\|_{k+1,\alpha,\gamma} \le C_{\eta}\|u\|_{k+2,\gamma,\alpha},$
\item[ii)] $\|F_{2}(u) - F_{2}(u')\|_{k,\alpha,\gamma} \le C_{\eta}\max\{\|u\|_{k+2,\alpha,\gamma},\|u'\|_{k+2,\alpha,\gamma}\}\|u-u'\|_{k+2,\alpha,\gamma}$, \newline $\|F_{2}(u)\|_{k,\alpha,\gamma} \le C_{\eta}\|u\|^{2}_{k+2,\alpha,\gamma}.$
\end{itemize}
\end{enumerate}
Then there exists a unique solution $u^{*} \in x^{\gamma}C^{k,\alpha}_{\Phi}(M\times [0,T_{0}])$ of the Cauchy problem \eqref{CP1}, for some $T_{0}>0$ sufficiently small.
\end{thm}
A generalization of Theorem \ref{cor-paper1} for some linear parabolic equation with non-constant coefficient will be presented here.
This will be achieved by a slight generalization of the mapping properties of $\mathbf{H}$ to the constructed parametrix for heat-type operators.

\section{Maximum principle for stochastically complete manifolds} \label{maxsection}

In order to construct a parametrix for the heat-type operator $\partial_t+a\Delta_\Phi$, we will employ a maximum principle.
We have seen in \S \ref{stoc-comp-subsect} that $\Phi$-manifolds are stochastically complete. 
A very neat property of stochastically complete manifolds, which is actually equivalent to stochastic completeness, is that they satisfy the Omori-Yau maximum principle. 
We begin this section by recalling the (strong) Omori-Yau maximum principle.
Afterwards we employ Omori-Yau to prove a parabolic maximum principle based on the first author's previous work \cite{bruno}.

\subsection{Omori-Yau maximum principle}\label{Omori-YauSection}

The Omori-Yau maximum principle for the Laplacian, defined in e.g. \cite[Definition 2.1]{alias2016maximum}, means that for any function $u \in C^{2}({{M}})$ with 
bounded supremum there is a sequence $\{p_{k}\}_{k} \subset {M}$ satisfying
\begin{equation}\label{omori-sup-strong}
    u(p_{k}) > \displaystyle\sup_{{M}}u - \dfrac{1}{k}, \quad |\nabla u(p_k)| \leq \dfrac{1}{k}, \quad - \Delta_{{g}}u(p_{k}) < \dfrac{1}{k}.
\end{equation}
Similarly, provided $u$ has bounded infimum, there exists a sequence $\{p'_{k}\}_{k} \subset {M}$ such that 
\begin{equation} \label{omori-inf-strong}
    u(p'_{k}) < \inf_{{M}} u + \dfrac{1}{k}, \quad |\nabla u(p'_k)| \leq \dfrac{1}{k}, \quad- \Delta_{{g}}u(p'_{k}) > \dfrac{1}{k}.
\end{equation}
As an example, by \cite{yau}, see also \cite[Theorem 2.3]{alias2016maximum}, the Omori-Yau maximum principle for the Laplacian holds for every complete Riemannian manifold $(M,g)$ with Ricci curvature bounded from below. 
We shall refer to this principle as the \emph{strong} Omori-Yau maximum
principle in order to distinguish it from another version of the principle on stochastically complete manifolds.

\begin{rmk}
We want to point out a difference with \cite{alias2016maximum} in the 
different sign convention for the Laplace-Beltrami operator.
\end{rmk}

According to Pigola, Rigoli and Setti in \cite[Theorem 1.1]{pigola} (see also \cite[Theorem 2.8 (i) and (iii)]{alias2016maximum}), a similar version of the Omori-Yau maximum principle holds for
stochastically complete manifolds. 
More precisely, for any $({M},g)$ satisfying e.g. the
volume growth condition in \eqref{vol-growth}, and any function $u \in C^{2}({M})$ bounded from above, there is a sequence $\{p_{k}\}_{k} \subset {M}$ such that
\begin{equation}\label{omori-sup}
    u(p_{k}) > \displaystyle\sup_{{M}}u - \dfrac{1}{k} \;\; \mbox{and} \;\; - \Delta_{{g}}u(p_{k}) < \dfrac{1}{k}.
\end{equation}
Similarly, if $u$ is bounded from below, there exists a sequence $\{p'_{k}\}_{k} \subset {M}$ such that 
\begin{equation} \label{omori-inf}
    u(p'_{k}) < \inf_{{M}} u + \dfrac{1}{k} \;\; \mbox{and} \;\; - \Delta_{{g}}u(p'_{k}) > \dfrac{1}{k}.
\end{equation}

\subsection{Classical H\"{o}lder spaces}\label{ClassicalHolderSpaceSec}

As mentioned above, if a Riemannian manifold $(M,g)$ is e.g. stochastically complete, then the Omori-Yau maximum principle in either of the formulations \eqref{omori-sup} and \eqref{omori-inf} hold for bounded functions.
For general non-compact manifolds one can not expect to be dealing with bounded functions.
Now, $\Phi$-manifolds are stochastically complete as discussed in \S \ref{stoc-comp-subsect} (see also \cite[\S 3]{paper1}).
Also, $\Phi$-manifolds can be thought as non-compact manifolds which are asymptotically conical (this can be achieved by performing a change of coordinates $r=1/x$ therefore "pushing" the boundary to infinity).
This means that one can not use Omori-Yau for any function.
But in \S \ref{HoldPhiSect} we have introduced some geometry-adapted H\"{o}lder spaces; in view of the H\"{o}lder norm defined in \eqref{alphanorm} one sees that $\Phi$-$k,\alpha$ H\"{o}lder functions are indeed bounded and, as a bonus, the heat-kernel is very well behaved as an operator between those spaces.
This leads to the following observation.
If a stochastically complete Riemannian manifold is given, then the Omori-Yau maximum principle would hold for functions living in some appropriate H\"{o}lder space.
Therefore here we give the classical definition of H\"{o}lder spaces and later (\S \ref{Maxsubsect}) we prove a parabolic Omori-Yau maximum principle for functions lying in such H\"{o}lder spaces.
As a remark, one can see that, in the setting of $\Phi$-manifolds, the geometry-adapted H\"{o}lder spaces (defined in \S \ref{HoldPhiSect}) are a subspace of the ones defined here; thus implying that the maximum principle presented in Theorem \ref{max-princ-thm} will also hold for $\Phi$-$k,\alpha$ H\"{o}lder functions.
\begin{defn} \label{holder}
Let $\alpha \in (0,1)$. 
We define the semi-norm 
\begin{equation}
    [u]_{\alpha} := \sup_{{M}^2_T} \left\{ \dfrac{|u(p,t)-u(p',t')|}{d(p,p')^{\alpha}+|t-t'|^{\alpha/2}}\right\},
\end{equation}
where the supremum is over ${M}^2_T$ with ${M}_T := {M} \times [0,T]$.  The distance $d$ is induced by the metric $g$.
The H\"{o}lder space $ C^{\alpha}({M}\times [0,T])$, is then defined as usual, as the space of continuous functions 
$u \in C^{0}({M}\times [0,T])$ with bounded $\alpha$-norm, that is
\begin{equation}\label{Classicalalphanorm}
    \|u\|_{\infty} + [u]_{\alpha}=:\|u\|_{\alpha} <\infty.
\end{equation}
\end{defn}

\medskip

Once equipped with the $\alpha$-norm \eqref{Classicalalphanorm}, the resulting normed vector space $C^{\alpha}({M}\times [0,T])$ is a Banach space.
Similarly one defines higher order H\"{o}lder spaces.
\begin{defn}\label{ClassicHolder}
Let $({M},g)$ be a Riemannian manifold and consider $k,l_1$ and $l_2$ to be non negative integers. 
We say that a function $u$ lies in $C^{k,\alpha}({M}\times[0,T])$ if $(P\circ \partial^{l_{2}}_{t})u$ lies in  $C^{\alpha}({M}\times [0,T]),$ for $P\in \diff^{l_1}({M}), \; 0\le l_{1} + 2l_{2} \le k$.  Here $\diff^{l_1}({M})$ denotes the space of differential operators of order $l_1$ over ${M}$. 
In particular, this is equivalent to require that the $(k,\alpha)$-norm, defined by
\begin{equation}
\|u\|_{k,\alpha}=\|u\|_\alpha+\displaystyle\sum_{l_{1} + 2l_{2} \le k} \sum_{P\in\diff^{l_1}({M})} \|(P\circ \partial_{t}^{l_{2}})u\|_{\alpha}
\end{equation}
is bounded.
\end{defn}
\begin{rmk}
By definition, we have the chain of inclusions $C^{l,\alpha}({M}\times[0,T])\subset C^{k,\alpha}({M}\times[0,T])$ for every $0\le k\le l$.
\end{rmk}

\subsection{Maximum principle}\label{Maxsubsect}

Based on the Omori-Yau maximum principle in \S \ref{Omori-YauSection}, the first named author jointly with Hartmann and Vertman proved the following enveloping theorem (cf. \cite{bruno}).
For convenience to the reader, we present the proof here as well. 

\begin{prop}\cite[Proposition 3.1]{bruno}\label{envelope}
Let $({M},g)$ be a stochastically complete manifold and consider $u\in C^{2,\alpha} ({M}\times[0,T])$. 
Then  the functions
$$u_{\sup} (t) := \sup_{{M}}u(\cdot, t), \quad u_{\inf}(t):= \inf_{{M}}u(\cdot, t)$$ are locally Lipschitz, hence differentiable
almost everywhere in $(0,T)$.
Moreover, at those differentiable times $t \in (0,T)$ we find
\begin{equation}\label{dini}
\begin{split}
&\frac{\partial}{\partial t}u_{\sup} (t) \leq \lim_{\ \epsilon \to \, 0^+} 
\left( \limsup_{k\to \infty} \frac{\partial u}{\partial t} \left(p_k(t+\epsilon), t + \epsilon\right)\right), \\
&\frac{\partial}{\partial t} u_{\inf} (t) \geq \lim_{\ \epsilon \to \, 0^+} 
\left( \liminf_{k\to \infty} \frac{\partial u}{\partial t} \left(p'_k(t+\epsilon), t + \epsilon\right)\right),  
\end{split}
\end{equation}
where $\left(p_k(t+\epsilon)\right)_k$ and $\left(p'_k(t+\epsilon)\right)$ are maximizing and minimizing sequences for the functions $u(\_,t+\epsilon)$ respectively as in \eqref{omori-sup} and \eqref{omori-inf}. 
\end{prop}
\begin{proof}
We begin by applying \eqref{omori-sup} to $u(t+\epsilon)$. 
Moreover, an application of the Mean Value Theorem leads to 
	\begin{equation*}
		u_{\sup} (t+\epsilon) \leq u(p_k(t+\epsilon), t) + \epsilon \cdot \frac{\partial u}{\partial t}(p_k(t+\epsilon), \xi) + \frac{1}{k},
	\end{equation*}
	for some $\xi \in (t,t+ \epsilon)$. 
Next we want to estimate $u_{\sup}(t+\epsilon)$ from below. 
By recalling that $u_{\sup}(t)\ge u\left(p_k(t+\epsilon),t\right)$ we get
	\begin{equation*}
		u_{\sup} (t+\epsilon) \geq u(p_k(t+\epsilon),t) + \epsilon \cdot 
		\frac{u_{\sup} (t+\epsilon) - u_{\sup} (t)}{\epsilon}. 
	\end{equation*}
	Combining the inequalities above, canceling the term $u(p_k(t+\epsilon),t)$ on each side and taking limit superior as $k \to \infty$ on the right hand side, we obtain
	\begin{equation*}
		\epsilon \cdot \frac{u_{\sup} (t+\epsilon) - u_{\sup} (t)}{\epsilon}
		\leq \epsilon \cdot \limsup_{k\to \infty}\frac{\partial u}{\partial t}(p_k(t+\epsilon), \xi).
	\end{equation*}
	Canceling $\epsilon$ on both sides, we find
	\begin{equation}\label{est1}
		\begin{split}
			\frac{u_{\sup} (t+\epsilon) - u_{\sup} (t)}{\epsilon}
			\leq &\limsup_{k\to \infty} \left( \frac{\partial u}{\partial t}(p_k(t+\epsilon), \xi)
			- \frac{\partial u}{\partial t}(p_k(t+\epsilon), t+\epsilon) \right)
			\\ &+ \limsup_{k\to \infty}\frac{\partial u}{\partial t}(p_k(t+\epsilon), t+\epsilon).
		\end{split}
	\end{equation}
	Since $u \in C^{2,\alpha}(M\times [0,T])$, we can estimate
	\begin{equation}\label{est2}
		\begin{split}
			&\bullet \quad  \limsup_{k\to \infty} \left| \frac{\partial u}{\partial t}(p_k(t+\epsilon), \xi)
			- \frac{\partial u}{\partial t}(p_k(t+\epsilon), t+\epsilon) \right| \leq \|u\|_{2,\alpha} \epsilon^{\alpha/2}, \\
			&\bullet \quad \limsup_{k\to \infty}\left| \frac{\partial u}{\partial t}(p_k(t+\epsilon), t+\epsilon) \right| \leq \|u\|_{2,\alpha}.
		\end{split}
	\end{equation}
	Hence, the two terms on the right-hand side in 
	\eqref{est1} are bounded uniformly in $\epsilon$. Now, after repeating the arguments with 
	the roles of $u_{\sup}(t)$ and $u_{\sup}(t+\epsilon)$ interchanged, we conclude that
	$u_{\sup}$ is locally Lipschitz.
Consequently, Rademacher's theorem implies that $u_{\sup}$ differentiable
	almost everywhere.
	
Let now $t\in (0,T)$ be one of the points at which $u_{\sup}$ is differentiable.  From \eqref{est1} and the first line in \eqref{est2}, by taking $\epsilon \to 0$ we conclude that
	\begin{equation}
		\begin{split}
			\frac{\partial}{\partial t} u_{\sup}(t) \leq \lim_{\epsilon \to 0} 
			\left( \limsup_{k\to \infty} \frac{\partial u}{\partial t} \left(p_k(t+\epsilon), t + \epsilon\right)\right). ,
		\end{split}
	\end{equation}
showing that the first inequality in \eqref{dini} holds. 
The second inequality follows from the first, using \eqref{omori-inf}, with $u$ replaced by $(-u)$.
\end{proof}
\noindent
We are now in the position to prove the claimed maximum principle.

\begin{thm}\label{maximumprinciple}(Theorem \ref{max-princ-thm})
Let $({M},g)$ be a $m$-dimensional stochastically complete manifold.
Furthermore, let $a\ge \delta>0$ be a bounded function on ${M}$. 
If $u \in C^{2,\alpha}({M}\times [0,T])$ is a solution of the Cauchy problem 
	\begin{equation}
		(\partial_{t} + a\Delta_{{g}})u = 0, \;\; u|_{t=0} = 0
	\end{equation}
then $u=0$.
\end{thm}

\begin{proof}
Since $u\in C^{2,\alpha}(M\times[0,T])$ it is, in particular, bounded for every $t$ meaning that $u(\_,t)$ is bounded.
Therefore we can find Omori-Yau maximizing and minimizing sequences $(p_k(t))$ and $(p'_k(t))$ satisfying \eqref{omori-sup} and \eqref{omori-inf}.
	Combining the first inequality in Proposition \ref{envelope} and \eqref{omori-sup}, it follows that
	\begin{equation*}
		\dfrac{\partial}{\partial t}u_{\sup}(t) \leq \lim_{\epsilon \to 0} 
		\left( \limsup_{k\to \infty} \frac{a(p_{k}(t+\epsilon),t+\epsilon)}{k}\right) \leq 0.
	\end{equation*}
	Analogously, by combining the second inequality in Proposition \ref{envelope} and \eqref{omori-inf}, we get 
	\begin{equation*}
		\dfrac{\partial}{\partial t}u_{\inf}(t) \geq \lim_{\epsilon \to 0} 
		\left( \liminf_{k\to \infty} \frac{-a(p_{k}(t+\epsilon),t+\epsilon)}{k}\right) \geq 0.
	\end{equation*}
	This means that the infimum of the function $u$ over $M$ is non-decreasing in time, while the supremum of the function $u$ over $M$ is non-increasing in time; since $u=0$ at time $t=0$, follows directly that $u=0$ on $M\times [0,T]$.
\end{proof}
The above result allows us to prove uniqueness of solutions to homogeneous and non-homogeneous linear heat-type Cauchy problems with variable coefficients.
\begin{cor}\label{UniquenessOfAny}
Denote by $P$ the heat-type operator $P=\partial_t+a\Delta_g$. 
If $u,v\in C^{2,\alpha}(M\times[0,T])$ are such that $Pu=Pv$ with $u\big|_{t=0}=v\big|_{t=0}$ then $u=v$.
\end{cor}
\begin{proof}
Note that $P$ is linear therefore, by setting $h=u-v$ we see that $h$ satisfies the Cauchy problem $(\partial_t+a\Delta_g)h=0$ with $h\big|_{t=0}=u\big|_{t=0}-v\big|_{t=0}=0$.
The above result implies $h=0$ resulting in $u=v$.
\end{proof}

\section{Parametrix construction for heat-type equations}\label{ParametrixSection}

We will now leave the more general setting of stochastically complete manifolds and move to the manifolds we are interested in, that is $\Phi$-manifolds.

For a given $\Phi$-manifold $(M,g_{\Phi})$, the heat-kernel operator $\mathbf{H}$ represents an inverse of the heat operator $(\partial_t+{\Delta})$.
Recall that here $\Delta$ denotes the unique self-adjoint extension of the Laplace-Beltrami operator associated to the $\Phi$-metric $g_\Phi$.
This means that, given some function $\ell\in x^\gamma C^{k,\alpha}_{\Phi}(M\times[0,T])$, $u=\mathbf{H}(\ell)$ is a solution of the Cauchy problem
\begin{equation} \label{CPE1}
   \left(\partial_t+{\Delta_{\Phi}}\right)u=\ell,\;\;u|_{t=0}=0. 
\end{equation}
The aim of this section is to get a similar result for heat-type operators
\begin{equation}\label{HeatTypeOp}
P:=\partial_t+a{\Delta_{\Phi}},
\end{equation}
where $a$ is a function on $\overline{M}\times[0,T]$.
Although not explicitly expressed here, the function $a$ will be subject to some restrictions (see Theorem \ref{theorem2}).

This will be accomplished by firstly constructing an approximate inverse, i.e. a parametrix, for the operator $P$.
By looking at heat-type operators as in \eqref{HeatTypeOp}, it is clear that the parametrix will be constructed by means of the standard heat-kernel operator $\mathbf{H}$.
Hence, by looking at Theorem \ref{MappingPropertiesTHM} one might expect to find "well behaved" parametrix for heat-type operators between the weighted H\"{o}lder spaces introduced in \S \ref{HoldPhiSect}.

\medskip
A parametrix for heat-type operators allows us to prove short-time existence of solutions to the following Cauchy problems
\begin{align}
Pu&=\left(\partial_t+a{\Delta}\right)u=\ell,\;\;u|_{t=0}=0,\\
Pu&=\left(\partial_t+a{\Delta}\right)u=0,\;\; u|_{t=0}=u_0,
\end{align}
for some functions $\ell:\overline{M}\times[0,T]\rightarrow\mathbb{R}$ and $u_0:\overline{M}\rightarrow\mathbb{R}$ respectively.
These last two statements are the core of Theorem \ref{theorem2} which we recall here for convenience of the reader.
\begin{thm}
Let $\beta$ be in $(0,1)$ and consider a positive function $a$ in $C^{k,\beta}_{\Phi}(M\times[0,T])$ to so that it is bounded from below away from zero.  There exist two operators $\mathbf{Q}$ and $\mathbf{E}$ so that, for every $\alpha\in (0,1)$, $\alpha<\beta$ and for every $\gamma\in \mathbb{R}$,
\begin{align*}
	&\mathbf{Q}:x^{\gamma}C^{k,\alpha}_{\Phi}(M\times [0,T])\rightarrow x^{\gamma}C^{k+2,\alpha}_{\Phi}(M\times [0,T]), \\
	&\mathbf{E}:x^{\gamma}C^{k,\alpha}_{\Phi}(M)\rightarrow x^{\gamma}C^{k+2,\alpha}_{\Phi}(M\times [0,T]),
\end{align*}
are both bounded.
Furthermore, for $\ell\in x^\gamma C^{k,\alpha}(M\times[0,T])$ and $u_0\in C^{k,\alpha}(M)$, $\mathbf{Q}\ell$ and $\mathbf{E}u_0$ are solutions of the Cauchy problems
\begin{align}
	\mbox{\textbf{(i)}}\, Pu = \ell; \;	u|_{t=0} = 0 \;\;	\mbox{and} \;\; \mbox{\textbf{(ii)}}\, Pu = 0;  \; u|_{t=0} = u_0
\end{align}
respectively.
\end{thm}
The construction of a parametrix will be split in two steps: a boundary parametrix and an interior parametrix.
A combination of those will then give rise to a parametrix for heat-type operators.
A boundary parametrix will be constructed in \S \ref{BoundaryParametrixSec}. 
Our construction follows along the same steps of the boundary parametrix in \cite{bahuaud2019long}. 
It is a technical construction since it requires a careful analysis near the boundary.
The construction of an interior parametrix, along with a parametrix for heat-type operators, will instead take place in \S \ref{InteriorParametrixSec}.
The interior parametrix will follow as a consequence of the standard analysis of parabolic PDE's on compact manifolds.
Proposition \ref{BoundariParametrixProp} will finally give us the parametrix of heat-type operators $P=\partial_t+a\Delta_\Phi$.
We will conclude this section with the proof of Theorem \ref{theorem2}.

\subsection{Boundary parametrix}\label{BoundaryParametrixSec}

As in \cite{bahuaud2019long}, the boundary parametrix will be constructed by localizing the problem in appropriate coordinate patches by making use of two partitions of unity.
Thus, we will firstly construct a localized parametrix, then by summing over the partition of unity, we get an approximate inverse of $P$ near the boundary.
The next Lemma explains the reason why the choice of partitions of unity, localized near the boundary, are useful for the purposes described at the beginning of this section. 
\begin{lem}\label{TechincalLemma}
Let ${(M,g_{\Phi})}$ be a $\Phi$-manifold and consider two functions $\varphi,\psi\in C^{\infty}(M)$ to be compactly supported.
Assume, furthermore, that $\varphi$ and $\psi$ lie in $C^{\alpha}_{\Phi}(M)$ (cf. \S \ref{HoldPhiSect}) and that $\psi$ is supported away from the boundary $\partial\overline{M}$ of $\overline{M}$. 
Let $\mathbf{H}$ be the heat-kernel operator described in \eqref{heatkerelop-def}.
Denote by $R^0$ the operator defined by $R^0=M(\psi)\circ \mathbf{H}\circ M(\varphi)$, i.e. $R^0 u=\psi \mathbf{H}(\varphi u)$.
Here $M(\psi)$ stands for the operator "multiplication by $\psi$".
For every non negative integer $k$, for every $\alpha\in (0,1)$ and $\gamma\in\mathbb{R}$, the operator $R^0$ acting between the weighted H\"{o}lder spaces  
$$R^0:x^\gamma C^{k,\alpha}_{\Phi}(M\times [0,T])\rightarrow\sqrt{t}x^\gamma\,C^{k+1,\alpha}_{\Phi}(M\times [0,T])$$
has operator norm $\|R^0\|_{\op}$ satisfying
$$\|R^0\|_{\op}\xrightarrow{T\rightarrow 0}0.$$
\end{lem}
\begin{proof}
With the same argument employed in the proof of \cite[Theorem 1]{paper1}, it is enough to prove the result for $k=0$.
It is important to point out that the operator $R^0$ acts as a convolution, i.e. for $u\in x^\gamma C^\alpha(M\times[0,T])$,
$$R^0u(p,t)=\int_0^t\int_M \psi (p)H(t-\widetilde{t},p,\widetilde{p})\varphi(\widetilde{p})u(\widetilde{p},\widetilde{t})\dvol_{\Phi}(\widetilde{p})\di \widetilde{t},$$
with $H$ being the heat-kernel whose asymptotics have been discussed in \cite[\S 5]{paper1}.
For simplicity we will denote the kernel of the operator $R^0$ just by $\psi H\varphi$.

\medskip
Since $\psi$ is supported away from the boundary $\partial\overline{M}$ of $\overline{M}$, the lift of $\psi H \varphi$ to the heat space ${{M}^2_h}$ is (compactly) supported away from ${\ff}$, ${\fd}$, ${\lf}$ and ${\rf}$ (see \cite[\S 4]{paper1}).
Therefore, according to \cite[\S 5]{paper1}, we conclude that the asymptotic behavior of $\psi H\varphi$ is given by the asymptotic of the operator $H$ near $\td$, that is
$$\beta^*(\psi H\varphi)\sim \tau^{-m}G_{1},$$
where $G_1$ is a bounded function vanishing to infinite order as $|(\mathcal{S},\mathcal{U},\mathcal{Z})|\rightarrow\infty$.

\medskip
In \cite[Theorem 6.1 and Theorem 6.2]{paper1} we have proven similar estimates for the heat-kernel operator $\mathbf{H}$. 
In that casae we have made use of the fact that the heat-kernel $H$ is "stochastically complete", meaning that it integrates to $1$.
Unfortunately, this is not the case here due to the presence of the functions $\psi$ and $\varphi$.
But estimating in projective coordinates and the above observation allow us to prove the claimed mapping properties.
In conclusion
\begin{align*}
\|R^0\|_{\op}=&\sup_{\|u\|_\alpha=1}\|R^0u\|_{1,\alpha}=\sup_{\|u\|_\alpha=1}\|R^0 u\|_\alpha +\sup_{\substack{\|u\|_\alpha=1\\ X\in\mathcal{V}_\Phi}}\|X(R^0u)\|_\alpha\le c\sqrt{t}.
\end{align*}
The above estimate implies the result since, for $T\rightarrow 0$, $\sqrt{t}\rightarrow 0$.
\end{proof}
We can now construct the specific partition of unity.

\subsubsection{\textbf{Partitions of unity}}\label{PartitionSec}

Let us fix some $\mathfrak{R}>0$ and consider the collar neighborhood $U_{\mathfrak{R}}=\left\{p\in \overline{M}\,|\, x(p)\le \mathfrak{R}\right\}$ of $\partial \overline{M}$ in $\overline{M}$.
Furthermore, for $d>0$ let us define the family of half-cubes
$$B(d)=[0,d)\times (-d,d)^b\times(-d,d)^f\subset \mathbb{R}_{\ge 0}\times \mathbb{R}^b\times\mathbb{R}^f,$$
where $b$ and $f$ denote the dimension of the closed manifolds $Y$ and $Z$ respectively.
Since $\overline{M}$ is a compact manifolds with boundary, every point $\overline{p}\in\partial \overline{M}$ admits some coordinate chart $\phi:B(1)\rightarrow A$, where $\overline{p} \in A$.
Moreover, due to compactness of $\partial \overline{M}$, we can consider finitely many charts $\{\overline{p}_i,\phi_i:B(1)\rightarrow A_i\}$ where the $\overline{p_i}$'s are points on the boundary $\partial \overline{M}$.
By choosing $\mathfrak{R}$ sufficiently small, the finite family $(A_i)_i$ will cover the whole collar neighborhood $U_{\mathfrak{R}}$.
Such a covering can be extended to a covering of the whole manifold $\overline{M}$ by considering an additional open set $A_0=\{p\in\overline{M}\,|\,x(p)>\mathfrak{R}/2\}$.   
\medskip

We will now define bump functions supported on the finite family of open neighborhoods of the points $\overline{p_i}\in\partial \overline{M}$.
We begin by setting $\sigma:\mathbb{R}_{\ge 0}\rightarrow\mathbb{R}$ to be a compactly supported function so that $\sigma(x)\le 1$, with $\sigma(x)=1$ for $x\in [0,1/2]$ and $\sigma(x)=0$ for $x\geq  1$.
Employing the Mean Value theorem, it is easy to see that $\sigma$ lies in $C^{k,\alpha}(\mathbb{R}_{\ge 0})$ for every $k\ge 0$ and for every $\alpha\in (0,1]$.
\begin{rmk}
The H\"{o}lder space $C^{k,\alpha}(\mathbb{R}_{\ge 0})$ above denotes the classical H\"{o}lder space, for which the H\"{o}lder bracket is defined by 
$$[\sigma]_{\alpha}=\sup_{x,x'\in\mathbb{R}_{\ge 0}}\frac{|\sigma(x)-\sigma(x')|}{|x-x'|^\alpha}.$$
\end{rmk}
Since it will play an important role later, we will stress here what happens to the function $\sigma$ after rescailing.
That is, if we consider some fixed number $\varepsilon\in (0,1)$ then we want to see what is the $\alpha$-H\"{o}lder  semi-norm of the function $\sigma\left(\frac{x}{\varepsilon}\right)$.
Since $\sigma$ lies in $C^\alpha(\mathbb{R}_{\ge 0})$ by definition then one readily sees that for every $x,x'\in \mathbb{R}_{\ge 0}$,
\begin{align*}
\left|\sigma\left(\frac{x}{\varepsilon}\right)-{\sigma}\left(\frac{x'}{\varepsilon}\right)\right|\le C\left|\frac{x}{\varepsilon}-\frac{x'}{\varepsilon}\right|^\alpha=C\varepsilon^{-\alpha}|x-x'|^\alpha,
\end{align*}  
thus implying that $\left[\sigma\left(\frac{x}{\varepsilon}\right)\right]_{\alpha}\le C\varepsilon^{-\alpha}$.
\medskip

Before proceeding with the definition of the bump functions, we need an intermediate result.
As it has been already done in \S \ref{PhiMfldsSect}, we will use the short hand notation $y$ and $z$ for $(y^1,\dots,y^b)$ and $(z^1,\dots, z^f)$ respectively.
\begin{lem}\label{DistanceEqLemma}
Let $(M,g_{\Phi})$ be a $\Phi$-manifold.
For every $q\in [1,\infty)$, the following distances on $M$ are equivalent:
\begin{equation*}
\begin{split}
\di_{q,\Phi}(p,p')&=\left(|x-x'|^q+(x+x')^q\|y-y'\|^q+(x+x')^2q\|z-z'\|^q\right)^{1/q},\\
\di_{\infty,\Phi}(p,p')&=\max\{|x-x'|,(x+x')\|y-y'\|,(x+x')^2\|z-z'\|\},
\end{split}
\end{equation*}
Here, by equivalent, we mean that for every $q,q'\in[1,\infty]$ there exist constants $c,C>0$ so that for every $p,p'\in M$, $$c\di_{q',\Phi}(p,p')\le \di_{q,\Phi}(p,p')\le C\di_{q',\Phi}(p,p').$$
\end{lem}
\begin{proof}
Notice that it is enough to prove that for a given $q\in [1,\infty)$, there exist constants $c,C>0$ so that 
$$c\di_{\infty,\Phi}(p,p')\le \di_{q,\Phi}(p,p')\le C \di_{\infty,\Phi}(p,p')$$
for every $p,p'\in M$.
Indeed one can use the transitive property to gain the other inequalities.  Thus, let us consider $q\in [1,\infty)$. 
For given $p,p'\in M$, it is straightforward that $\di_{\infty,\Phi}(p,p')\le \di_{q,\Phi}(p,p')$. 
The other inequality follows by arguing as follows:
\begin{align*}
    \di_{q,\Phi}(p,p')&\le \left(\di_{\infty,\Phi}(p,p')^q+\di_{\infty,\Phi}(p,p')^q+\di_{\infty,\Phi}(p,p')^q\right)^{1/q}\\
    &=3^{1/q}\di_{\infty,\Phi}(p,p').
\end{align*}
\end{proof}
We are now in the position to define the appropriate bump functions.
Let $\overline{p}\in\partial \overline{M}$ be fixed.
From the definition of the open covering, defined above, there exists some $\phi_i:B(1)\rightarrow A_i$ so that $\phi_i(\overline{p})=(0,\overline{y},\overline{z})$ for some $\overline{y}\in (-1,1)^b$ and $\overline{z}\in (-1,1)^f$.
\begin{prop}\label{PartitionOfUnityProp}
Let $\varepsilon\in (0,1)$ be fixed.
For $p \in A_{i}$, with $\phi^{-1}_{i}(p) = (x,y,z)$, consider the functions $\widehat{\psi}_{i,\overline{p}},\widehat{\varphi}_{i,\overline{p}}:A_i\rightarrow\mathbb{R}$ defined by
\begin{align*}
\widehat{\varphi}_{i,\overline{p}}(p)=&\sigma\left(\frac{x}{\varepsilon}\right)\sigma(x\|y-\overline{y}\|)\sigma(\varepsilon x^2\|z-\overline{z}\|),\\
\widehat{\psi}_{i,\overline{p}}(p)=&\sigma\left(\frac{x}{2\varepsilon}\right)\sigma\left(\frac{x\|y-\overline{y}\|}{2}\right)\sigma\left(\frac{\varepsilon x^2\|z-\overline{z}\|}{2}\right).
\end{align*}
Then $\widehat{\varphi}_{i,\overline{p}}$ and $\widehat{\psi}_{i,\overline{p}}$ satisfy: 
\begin{itemize}
\item[I.] $\widehat{\psi}_{i,\overline{p}}\equiv 1$ on the support of $\widehat{\varphi}_{i,\overline{p}}$. 
\item[II.] There exists constants (all of which will be denoted by $C$) so that  $[\widehat{\psi}_{i,\overline{p}}]_{\alpha}\le C\varepsilon^{-\alpha}$, $[\widehat{\varphi}_{i,\overline{p}}]_{\alpha}\le C\varepsilon^{-\alpha}$.
\item[III.]$\widehat{\varphi}_{i,\overline{p}},\widehat{\psi}_{i,\overline{p}}\in C^{2,\alpha}_{\Phi}(M)$ (see \S \ref{HoldPhiSect} for the definition of H\"{o}lder spaces on $\Phi$-manifolds).
\item[IV.] There exists some constant $\overline{C}>0$ (depending solely on the dimension of $Y$ and $Z$) so that $\diam\left(\supp(\widehat{\varphi}_{{i,\overline{p}}})\right)\le \overline{C}\varepsilon$.
Here by $\diam$ we mean the diameter, that is 
$$\diam\left(\supp(\widehat{\varphi}_{{i,\overline{p}}})\right)=\max_{p,p'\in \supp\left(\widehat{\varphi}_{{i,\overline{p}}}\right)}\di_{\Phi}(p,p').$$
\end{itemize} 
Note that in the above we did not specify whit respect to which of the distances on $\overline{M}$ is the diameter considered since, by Lemma \ref{DistanceEqLemma} they are all equivalent.
\end{prop}
\begin{proof}
Property I follows directly from the definition of $\sigma$.
\medskip

Clearly, the fact that $\widehat{\varphi}_{i,\overline{p}}$ and $\widehat{\psi}_{i,\overline{p}}$ lie in $C^{\alpha}_{\Phi}(M)$ is a direct consequence of II, due to $\widehat{\psi}_{i,\overline{p}}$ and $\widehat{\varphi}_{i,\overline{p}}$ being bounded.
Let us therefore prove II.
Since $\widehat{\psi}_{i,\overline{p}}$ is just a rescaling of $\widehat{\varphi}_{i,\overline{p}}$, it is enough to prove II for the function $\widehat{\varphi}_{i,\overline{p}}$.
\medskip

Let $p,p'\in A_i$ and assume $\phi_i^{-1}(p)=(x,y,z)$ while $\phi_i^{-1} (p')=(x',y',z')$.
We have the following chain of inequalities:
\begin{align*}
\widehat{\varphi}_{i,\overline{p}}(p)-\widehat{\varphi}_{i,\overline{p}}(p')=&\sigma\left(\frac{x}{\varepsilon}\right)\sigma(x\|y-\overline{y}\|)\sigma(\varepsilon x^2\|z-\overline{z}\|)\\
-&\sigma\left(\frac{x'}{\varepsilon}\right)\sigma(x'\|y'-\overline{y}\|)\sigma(\varepsilon x'^2\|z'-\overline{z}\|)\\
\le& C\varepsilon^{-\alpha}|x-x'|^\alpha+C\left(\sigma(x\|y-\overline{y}\|)-\sigma(x'\|y'-\overline{y}\|)\right)\\
+&C\left(\sigma(\varepsilon x^2\|z-\overline{z}\|)-\sigma(\varepsilon x'^2\|z'-\overline{z}\|)\right)\\
\le& C\varepsilon^{-\alpha}|x-x'|^\alpha+C\left((x-x')\|y-\overline{y}\|+x'\|y-\overline{y}\|-x'\|y'-\overline{y}\|\right)^\alpha\\
+&C\varepsilon^{\alpha}\left(x^2\|z-\overline{z}\|-x'^2\|z-\overline{z}\|+x'^2\|z-\overline{z}\|-x'^2\|z'-\overline{z}\|\right)^\alpha\\
\le& C\varepsilon^{-\alpha}|x-x'|^\alpha+C\left(|x-x'|\|y-\overline{y}\|\right)^\alpha+C\left(x'\|y-y'\|\right)^\alpha\\
+&C\varepsilon^{\alpha}\left(|x^2-x'^2|\|z-\overline{z}\|\right)^\alpha+C\varepsilon^{\alpha}\left(x'^2\|z-z'\|\right)^\alpha\\
\le&C\varepsilon^{-\alpha}|x-x'|^\alpha+C\varepsilon^{-\alpha}\left(|x-x'|\|y-\overline{y}\|\right)^\alpha+C\varepsilon^{-\alpha}\left(x'\|y-y'\|\right)^\alpha\\
+&C\varepsilon^{-\alpha}\left(|x-x'|(x+x')\|z-\overline{z}\|\right)^\alpha+C\varepsilon^{-\alpha}\left(x'^2\|z-z'\|\right)^\alpha\\
\le &C\varepsilon^{-\alpha}|x-x'|^\alpha+C\varepsilon^{-\alpha}\left(x'\|y-y'\|\right)^\alpha+C\varepsilon^{-\alpha}\left(x'^2\|z-z'\|\right)^\alpha\\
\le &C\varepsilon^{-\alpha}|x-x'|^\alpha+C\varepsilon^{-\alpha}\left(x'\|y-y'\|+x\|y-y'\|\right)^\alpha\\
+&C\varepsilon^{-\alpha}\left(x'^2\|z-z'\|+(2xx'+x^2)\|z-z'\|\right)^\alpha\\
\le& C\varepsilon^{-\alpha}\left(|x-x'|^\alpha+(x+x')^\alpha\|y-y'\|^\alpha+(x+x')^{2\alpha}\|z-z'\|^\alpha\right)\\
\le& C\varepsilon^{-\alpha}\di_{\infty,\Phi}(p,p')^\alpha \le C\di_{2,\Phi}(p,p')^\alpha.
\end{align*}
It is important to mention that the $C$'s in the above estimate represent (perhaps different) uniform constants.
Note that the third inequality is obtained by making use of the reverse triangle inequality and sublinearity of $x^\alpha$ (with $\alpha\in (0,1)$).
The fourth inequality follows from the inequalities $\varepsilon^\alpha\le 1\le \varepsilon^{-\alpha}$ while the fifth inequality is a direct consequence of $\|y-\overline{y}\|$, as well as $|x+x'|$ and $\|z-\overline{z}\|$, being bounded.
So far, we have seen $\widehat{\varphi}_{i,\overline{p}}$ and $\widehat{\psi}_{i,\overline{p}}$ to lie in $C^{\alpha}_{\Phi}(M)$. 
This result can be extended to $C^{2,\alpha}_{\Phi}(M)$ just by noticing $\sigma$ to be constant near $\overline{p}$.
\medskip

Finally, let us prove IV. 
Consider $p,p'\in\supp(\widehat{\varphi}_{i,\overline{p}})$ with $\phi_i^{-1}(p)=(x,y,z)$ and $\phi_i^{-1}(p')=(x',y',z')$.
From the definition of $\sigma$, it is known that $x,x' \in (0,\varepsilon]$. 
Thus, computing $\di_{1,\Phi}(p,p')$ we get
\begin{align*}
\di_{1,\Phi}(p,p')=&|x-x'|+(x+x')\|y-y'\|+(x+x')^2\|z-z'\|\\
&\le \varepsilon+4\sqrt{b}\varepsilon+4\sqrt{f}\varepsilon^2\le \overline{C}\varepsilon
\end{align*}
with $\overline{C}=\max\{1,4\sqrt{b},4\sqrt{f}\}$.
Notice that the values $2\sqrt{b}$ and $2\sqrt{f}$ come from the Euclidean length of the diagonal of the cubes $(-1,1)^b$ and $(-1,1)^f$ respectively. 
\end{proof}
\begin{rmk}
We want to point out that the function $\widehat{\varphi}_{i,\overline{p}}$ and $\widehat{\psi}_{i,\overline{p}}$ from Proposition \ref{PartitionOfUnityProp} are defined on the open sets $A_i$. 
Due to the nature of the function $\sigma$, it is possible to extended each of them to the entire manifold $\overline{M}$ by making them vanish outside their respective supports.  With a slight abuse of notation, we shall not distinguish $\widehat{\varphi}_{i,\overline{p}}$ and $\widehat{\psi}_{i,\overline{p}}$ from their extensions.

Moreover, it is worth pointing out that the functions $\widehat{\varphi}_{i,\overline{p}}$ and $\widehat{\psi}_{i,\overline{p}}$ are, in fact, far more regular than simply $C^2$.  Non the less, property III is stated only for $C^2$ due to extra negative powers of $\varepsilon$ appearing in estimating the $\alpha$-seminorm of the derivatives.
\end{rmk}
The functions $\widehat{\varphi}_{i,\overline{p}}$ and $\widehat{\psi}_{i,\overline{p}}$ will allow us to construct the claimed partitions of unity.
Recall that, for a partition of unity, only a finite number of functions may be non-vanishing in a neighborhood.
Although we have a finite family of open sets $(A_i)_i$, the functions $\widehat{\varphi}_{i,\overline{p}}$, $\widehat{\psi}_{i,\overline{p}}$ are defined for every point on the boundary $\partial\overline{M}$ of $\overline{M}$.
This makes virtually impossible to have only finitely many non-vanishing functions in neighborhoods of points in a collar neighborhood of the boundary.
Hence, the final step for the construction of partitions of unity is to "reduce" the amount of points $\overline{p}\in\partial\overline{M}$ by means of which we defined the bump functions $\widehat{\varphi}_{i,\overline{p}}$ and $\widehat{\psi}_{i,\overline{p}}$.
To this end, let us consider the following set: for a fixed $\vartheta\in (0,1)$ consider
$$E_{i,\vartheta}=A_i\cap \left\{\phi_i(0,\vartheta\Lambda)\,|\,\Lambda\in\mathbb{Z}^{b+f}\right\}.$$
Recall that $\phi_i:B(1)\rightarrow A_i$ is a diffeomorphism, thus the set $E_{i,\vartheta}$ consists of finitely many boundary points in $A_i$.
This especially means that the family of functions $(\widehat{\varphi}_{i,\overline{p}})_{i,\overline{p}\in E_{i,\vartheta}}$, as well as for the family $(\widehat{\psi}_{i,\overline{p}})_{i,\overline{p}\in E_{i,\vartheta}}$, are finite.
\begin{rmk}\label{SupportOfPartitionRMK}
By definition of $\sigma$ we can conclude that there exists an open neighborhood of the boundary $\partial \overline{M}$ of $\overline{M}$, contained in the collar neighborhood $U_{\mathfrak{R}}$, so that every point $q$ in such a neighborhood lies in the support of at most finitely many of the functions $(\widehat{\varphi}_{i,\overline{p}})_{i,\overline{p}\in E_{i,\vartheta}}$ and $(\widehat{\psi}_{i,\overline{p}})_{i,\overline{p}\in E_{i,\vartheta}}$.
\end{rmk}
The only thing left to get partitions of unity (on an open neighborhood of the boundary) is to let the families $(\widehat{\varphi}_{i,\overline{p}})$ and $(\widehat{\psi}_{i,\overline{p}})$ to sum up to $1$. 
This is achieved by a trivial "normalization", to this end, for some $\vartheta\in (0,1)$ and for $\overline{p}\in E_{i,\vartheta}$, we define the functions $\varphi_{i,\overline{p}}$ and $\psi_{i,\overline{p}}$ as follows:
\begin{equation}\label{PartitionDefSupp}
\varphi_{i,\overline{p}}(p):=\frac{\widehat{\varphi}_{i,\overline{p}}(p)}{\sum_j\sum_{\overline{p}\in E_{j,\vartheta}}\widehat{\varphi}_{j,\overline{p}}(p)}\;\;\mbox{ and }\:\:\psi_{i,\overline{p}}(p):=\frac{\widehat{\psi}_{i,\overline{p}}(p)}{\sum_j\sum_{\overline{p}\in E_{j,\vartheta}}\widehat{\psi}_{j,\overline{p}}(p)}.
\end{equation}
It is now clear that both families $(\varphi_{i,\overline{p}})_{i,\overline{p}\in E_{i,\vartheta}}$ and $(\psi_{i,\overline{p}})_{i,\overline{p}\in E_{i,\vartheta}}$ are partition of unity on open neighborhoods of $\partial\overline{M}$.
Furthermore, since \eqref{PartitionDefSupp} holds only for points contained in the support of some of the functions $\widehat{\varphi}_{i,\overline{p}}$ and $\widehat{\psi}_{i,\overline{p}}$, it follows that properties I to IV in Proposition \ref{PartitionOfUnityProp} hold for the families $(\varphi_{i,\overline{p}})_{i,\overline{p}}$ and $(\psi_{i,\overline{p}})$.
\begin{rmk}
Notice that the functions $\widehat{\varphi}_{i,\overline{p}}$ and $\widehat{\psi}_{i,\overline{p}}$ are defined in terms of some $\varepsilon\in (0,1)$.
Thus the families $(\varphi_{i,\overline{p}})_{i,\overline{p}}$ and $(\psi_{i,\overline{p}})_{i,\overline{p}}$ are partitions of unity for any choice of $\varepsilon$.
\end{rmk}
Finally, the function 
\begin{equation}\label{phisum}
\phi:=\sum_i\sum_{\overline{p}\in E_{i,\vartheta}}\varphi_{i,\overline{p}}
\end{equation}
is constantly equal to $1$ on an open neighborhood of $\partial\overline{M}$ and satisfies properties I to IV in Proposition \ref{PartitionOfUnityProp} as well.

\subsubsection{\textbf{Boundary Parametrix}}\label{BoundaryParametrixSec2}

The partitions of unity presented in \S \ref{PartitionSec} allow us to construct a boundary parametrix for heat-type operators $P$ (cf. \eqref{HeatTypeOp}).
\medskip

Let $\gamma\in\mathbb{R}$ and $\alpha\in (0,1)$ be fixed and consider $\ell\in x^\gamma C^{\alpha}_{\Phi}(M\times[0,T])$.
A parametrix for an heat-type operator $P$ is a map $\mathbf{Q}:x^\gamma C^{\alpha}_{\Phi}(M\times[0,T])\rightarrow x^\gamma C^{2,\alpha}_{\Phi}(M\times [0,T])$ so that $u=\mathbf{Q}\ell$ is a solution for the parabolic Cauchy problem
\begin{equation}\label{HeatTypeEq}
Pu=(\partial_t+a\Delta)u=\ell;\;\;u|_{t=0}=0.
\end{equation}
Our first step towards the construction of $\mathbf{Q}$ is establishing an operator $\mathcal{Q}_{B}:x^\gamma C^{\alpha}_{\Phi}(M\times[0,T])\rightarrow x^\gamma C^{2,\alpha}_{\Phi}(M\times [0,T])$ giving rise to approximate solutions of \eqref{HeatTypeEq} near the boundary (the notion of approximate solutions is in the spirit of Lemma \ref{LocalSolvabilityBoundaryLemma} below). 
Hence, in order to do this, we localize \eqref{HeatTypeEq} near the boundary.
Let us therefore fix some $\overline{p}\in\partial\overline{M}$. 
As pointed out in Remark \ref{SupportOfPartitionRMK} every point on the boundary lies in the support of at most finitely many of the functions defined in \eqref{PartitionDefSupp}. 
Thus, without loss of generality, we can assume $\overline{p}$ to lie in some $E_{i,\vartheta}$ for some $i$ and some $\vartheta\in (0,1)$, which, from now on, will be considered to be fixed.
Next we freeze the coefficient $a$ of the Laplace-Beltrami operator at $t=0$.
In particular we focus our attention to the parabolic Cauchy problem with constant coefficient
\begin{equation}\label{cauchyproblem}
    P(\overline{p},0)\overline{u}_{\overline{p}}:=(\partial_t+a(\overline{p},0)\Delta)\overline{u}_{\overline{p}}=\varphi_{i,\overline{p}}\ell,\;\;
    \overline{u}_{\overline{p}}|_{t=0}=0.
\end{equation}

\medskip
Note that the Cauchy problem \eqref{cauchyproblem} is formally different from the Cauchy problem in \eqref{HeatTypeEq}, not only due to the localization but especially because the coefficient $a$ of the Laplace-Beltrami operator is now constant.

\medskip
By assuming $a$ to be positive and bounded from below away from zero, it is clear that, upon rescaling, the heat-kernel operator of $a(\overline{p},0)\Delta$, denoted by $\mathbf{H}_{\gamma,\overline{p}}$, is the same as the one from \S \ref{heatkerelmap-subsect}.
It follows that a solution for \eqref{cauchyproblem} is given by $\mathbf{H}_{\gamma,\overline{p}}(\varphi_{i,\overline{p}}\ell)$.
In particular, by defining
\begin{equation}\label{ubarpbar}
u_{\overline{p}}=\mathcal{Q}_{\gamma,i,\overline{p}}(\ell):=\psi_{i,\overline{p}}\mathbf{H}_{\gamma,\overline{p}}(\varphi_{i,\overline{p}}\ell),
\end{equation}
we have the following:
\begin{lem}\label{LocalSolvabilityBoundaryLemma}
Let $\alpha<\beta\le 1$ and assume $a \in C^{\beta}_{\Phi}(M\times [0,T])$ to be positive and bounded from below away from zero. 
Then, for every $\ell\in C^{\alpha}_{\Phi}(M\times[0,T])$, the function $u_{\overline{p}}=\mathcal{Q}_{\gamma,i,\overline{p}}(\ell)$, defined in \eqref{ubarpbar}, satisfies
\begin{equation}\label{Error1}
    Pu_{\overline{p}}:=(\partial_t+a\Delta)u_{\overline{p}}=\varphi_{i,\overline{p}}\ell+ R_{i,\overline{p}}^1\ell+R^2_{i,\overline{p}}\ell
\end{equation}
where
\begin{itemize}
    \item [a)] $R^1_{i,\bar{p}}:x^\gamma C^{\alpha}_\Phi(M\times [0,T])\rightarrow x^\gamma C^{\alpha}_\Phi(M\times[0,T])$ is a bounded operator.
Moreover, if $T<1$  there exists some constant $C>0$ so that
    $$\|R^1_{i,\bar{p}}\ell\|_{\alpha}\le C\left(T^{\alpha/2}+\varepsilon^{\beta}\right)\varepsilon^{-\alpha}\|\ell\|_{\alpha}.$$
    \item[b)] $R^2_{i,\bar{p}}:x^\gamma C^{\alpha}_\Phi(M\times [0,T])\rightarrow x^\gamma C^{\alpha}_\Phi(M\times[0,T])$ is a bounded operator and its operator norm goes to $0$ as $T\rightarrow 0^+$ i.e.
    $$\lim_{T\rightarrow 0^+}\|R^2_{i,\bar{p}}\|_{\op}=0.$$
\end{itemize}
\end{lem}
\begin{proof}
In order to avoid the plethora of indices we will suppress all the indices on $\varphi,\psi$ and the error terms $R^0$ and $R^1$.
Following the same computations as in \cite[Lemma 4.3]{bahuaud2014yamabe} one gets
\begin{align}
Pu_{\overline{p}}=&\psi\partial_t\mathbf{H}_{\overline{p}}(\varphi \ell)+[a\Delta,\psi]\left(\mathbf{H}_{\overline{p}}(\varphi \ell)\right)+\psi a\Delta \left(\mathbf{H}_{\overline{p}}(\varphi \ell)\right)\nonumber\\
=&\psi(\partial_t+a\Delta)\left(\mathbf{H}_{\overline{p}}(\varphi \ell)\right)+[a\Delta,\psi]\left(\mathbf{H}_{\overline{p}}(\varphi \ell)\right)\nonumber\\
=&\psi\left(\partial_t+a(\overline{p},0)\right)\Delta)\left(\mathbf{H}_{\overline{p}}(\varphi \ell)\right)+\psi\left(a-a(\overline{p},0)\Delta\right)\left( \mathbf{H}_{\overline{p}}(\varphi \ell)\right)\label{BPCTDE}\\
&+[a\Delta,\psi]\left(\mathbf{H}_{\overline{p}}(\varphi \ell)\right)\nonumber\\
=&\psi \varphi \ell+\psi\left(a-a(\overline{p},0)\right)\Delta\left( \mathbf{H}_{\overline{p}}(\varphi\ell)\right)+[a\Delta,\psi]\left(\mathbf{H}_{\overline{p}}(\varphi \ell)\right)\nonumber\\
=&:\psi\varphi\ell+R^1\ell+R^2\ell=\varphi\ell +R^1\ell+R^2\ell;\nonumber
\end{align}
where $[a\Delta,\psi]$ denotes the commutator between the differential operators $a\Delta$ and the "multiplication by $\psi$" operator.
Note that the fourth equality in \eqref{BPCTDE} follows from  $\mathbf{H}_{\overline{p}}(\varphi f)$ being a solution of the localized Cauchy problem.
Moreover, the last equality is a consequence of property I in Proposition \ref{PartitionOfUnityProp}.

\medskip
We will estimate the norms of $R^1$ and $R^2$ with $\gamma=0$; the case for generic $\gamma$ is slightly more involved but it follows along the same lines.
Furthermore, the estimates will be performed on $\supp(\psi)$ since the $\alpha$-norm is not effected by such a change.
\medskip

Let us begin by estimating the $\alpha$-norm of the operator $R^1$ applied to the function $\ell$.
\begin{equation}\label{R1EstimateFirst}
\begin{split}
\|R^1\ell\|_{\alpha}=&\|R^1\ell\|_{\infty}+[R^1\ell]_{\alpha}\\
\le&\|\psi\|_{\infty}\|a-a(\overline{p},0)\|_{\infty}\|\Delta\mathbf{H}(\varphi\ell)\|_{\infty}\\
&+[\psi]_{\alpha}\|a-a(\overline{p},0)\|_{\infty}\|\Delta\mathbf{H}(\varphi\ell)\|_{\infty}\\
&+\|\psi\|_{\infty}[a-a(\overline{p},0)]_{\alpha}\|\Delta\mathbf{H}(\varphi\ell)\|_{\infty}\\
&+\|\psi\|_{\infty}\|a-a(\overline{p},0)\|_{\infty}[\Delta\mathbf{H}(\varphi\ell)]_{\alpha}.
\end{split}
\end{equation}
We will estimate each term in \eqref{R1EstimateFirst} separately.
In what follows, unless otherwise specified, we will denote all the uniform constants by $C$.

\medskip
We begin by estimating the first term in \eqref{R1EstimateFirst}.
By assumption $a\in C^{\beta}_{\Phi}(M\times[0,T])$ with $\beta>\alpha$.
Thus one deduces
\begin{equation}\label{SupNorma-a0}
\|a-a(\overline{p},0)\|_{\infty}\le C(\varepsilon^\beta+T^{\beta/2});
\end{equation}
for some constant $C>0$, due to property IV in Proposition \ref{PartitionOfUnityProp}. 
From Theorem \ref{SecondMappingPropertyTHM} one has boundedness of the operator
$\Delta\mathbf{H}:C^{\alpha}_{\Phi}(M\times[0,T])\rightarrow t^{\alpha/2}C^{0}_{\Phi}(M\times[0,T])$; thus resulting in the estimate
\begin{equation}\label{SupNormDeltaH}
\|\Delta\left(\mathbf{H}(\varphi\ell)\right)\|_\infty\le Ct^{\alpha/2}\|\varphi\ell\|_\infty\le CT^{\alpha/2}\|\ell\|_{\infty}\le CT^{\alpha/2}\|\ell\|_{\alpha}.
\end{equation}
Hence the first term in \eqref{R1EstimateFirst} can be estimated by
\begin{equation}\label{R1EstimateSup}
\|\psi\|_{\infty}\|a-a(\overline{p},0)\|_{\infty}\|\Delta\left(\mathbf{H}(\varphi\ell)\right)\|_{\infty}\le C(\varepsilon^\beta+T^{\beta/2})T^{\alpha/2}\|\ell\|_{\alpha}.
\end{equation}
For the second term in \eqref{R1EstimateFirst} we use property II in Proposition \ref{PartitionOfUnityProp} paired with \eqref{SupNorma-a0} and \eqref{SupNormDeltaH}, resulting in 
\begin{equation}\label{R1EstimateSecond}
[\psi]_{\alpha}\|a-a(\overline{p},0)\|_{\infty}\|\Delta\left(\mathbf{H}(\varphi\ell)\right)\|_{\infty}\le C\varepsilon^{-\alpha}T^{\alpha/2}(\varepsilon^{\beta}+T^{\beta/2}).
\end{equation}
The third term in \eqref{R1EstimateFirst} can be estimated by noticing the following.
Recall that we are estimating on the $\supp(\psi)$; thus, from property IV in Proposition \ref{PartitionOfUnityProp}, for very $p,p'$ lying in the support of $\psi$, $\di_{\Phi}(p,p')\le \overline{C}\varepsilon$.
By choosing $\varepsilon$ small enough, e.g. $\varepsilon\le 1/\overline{C}$, and $T<1$, which will be consistent for our future applications (cf. Proposition \ref{BoundariParametrixProp}), we have
$$\di_{\Phi}(p,p')^\beta+|t-t'|^{\beta/2}\le \di_{\Phi}(p,p')^\alpha+|t-t'|^{\alpha/2}.$$
This implies, due to the assumption $a\in C^{\beta}_{\Phi}(M\times[0,T])$ and thus $[a]_{\beta}\le C$, that $[a-a(\overline{p},0)]_{\alpha}=[a]_{\alpha}\le C$.
Therefore we find 
\begin{equation}\label{R1EstimateThird}
\|\psi\|_{\infty}[a-a(\overline{p},0)]_{\alpha}\|\Delta\left(\mathbf{H}(\varphi\ell)\right)\|_{\infty}\le CT^{\alpha/2}\|\ell\|_{\alpha}.
\end{equation}
Finally, in order to estimate the fourth, and last, term in \eqref{R1EstimateFirst} we use the mapping property discussed in Theorem \ref{MappingPropertiesTHM} to deduce $\Delta \mathbf{H}:C^\alpha_{\Phi}(M\times[0,T])\rightarrow C^\alpha_{\Phi}(M\times[0,T])$ to be bounded.
Thus
$$[\Delta\left(\mathbf{H}(\varphi\ell)\right)]_{\alpha}\le \|\Delta\left(\mathbf{H}(\varphi\ell)\right)\|_{\alpha}\le C\|\varphi\ell\|_{\alpha}\le C\varepsilon^{-\alpha}\|\ell\|_{\alpha}$$
which, in turn, implies
\begin{equation}\label{R1EstimateFourth}
\|\psi\|_{\infty}\|a-a(\overline{p},0)\|_{\infty}[\Delta\left(\mathbf{H}(\varphi\ell)\right)]_{\alpha}\le C(\varepsilon^{\beta}+T^{\beta/2})\varepsilon^{-\alpha}\|\ell\|_{\alpha}.
\end{equation}
Joining \eqref{R1EstimateSup}-\eqref{R1EstimateFourth} together, in view of \eqref{R1EstimateFirst}, we conclude
\begin{align*}
\|R^1\ell\|_{\alpha}\le&C\left(T^{\alpha/2}(\varepsilon^{\beta}+T^{\beta/2})+\varepsilon^{-\alpha}T^{\alpha/2}(\varepsilon^{\beta}+T^{\beta/2})+T^{\alpha/2}+\varepsilon^{-\alpha}(\varepsilon^{\beta}+T^{\beta/2})\right)\|\ell\|_{\alpha}\\
\le&C\left(T^{\alpha/2}+\varepsilon^{\beta}\right)\varepsilon^{-\alpha}\|\ell\|_{\alpha};
\end{align*}
where the $C$'s denote different uniform constants.
We want to point out that the estimate above holds due to $\varepsilon\le 1/\overline{C}<1$ and $0<\alpha<\beta\le 1$, concluding the first part of the statement.
\medskip

For the second part we argue as follows.
By making use of the product rule one sees that for every twice-differentiable function$w$,
$$[a\Delta,\psi]w=a\Delta(\psi)\cdot w-2a\;g_{\Phi}(\nabla\psi,\nabla w),$$
where $\nabla$ denotes the gradient.
Note that our choice of $\psi$ implies that all of its derivatives are vanishing near the boundary $\partial\overline{M}$.
Thus, by choosing $w=\mathbf{H}_{\overline{p}}(\varphi\ell)$, we see that the assumption of Lemma \ref{TechincalLemma} are satisfied. 
Hence, $R^2:C^{\alpha}_{\Phi}(M\times[0,T])\rightarrow C^{\alpha}_{\Phi}(M\times[0,T])$ is a bounded operator with operator norm converging to $0$ as $T\rightarrow 0^+$.
\end{proof}
\begin{rmk}
We want to point out the main difference between the result presented here and the analogous result for edge manifolds \cite[Lemma 4.3]{bahuaud2014yamabe}. 
In \cite{bahuaud2014yamabe} the authors use the Mean Value Theorem to estimate the supremum norm of the coefficient $a$ of the Laplace-Beltrami operator.
This leads to terms which can be estimate against the incomplete edge distance.
In particular they reach an estimate of the form 
$$\|a-a(\overline{p},0)\|_{\infty}\le C(\varepsilon+T^{\alpha/2})$$
for some positive constant $C$ (cf. \cite[page 21]{bahuaud2014yamabe}).  
In our case, an application of the Mean Value Theorem does not lead to something comparable with the $\Phi$-distance $\di_{\Phi}$.
Therefore, we could assume less regularity from a differentiability point of view.
But the assumption $a\in C^{\alpha}_{\Phi}(M\times[0,T])$ is not enough to guarantee the existence of a boundary parametrix (see Proposition \ref{BoundariParametrixProp}).
Indeed one can see that, by assuming $a\in C^{\alpha}_{\Phi}(M\times[0,T])$, the estimates performed in the proof of Lemma \ref{LocalSolvabilityBoundaryLemma} lead to 
$$\|R^1\ell\|_{\alpha}\le C(T^{\alpha/2}\varepsilon^{-\alpha}+1)$$
which, in turn, can not be made less than one thus making it impossible for $R^1$ to have small operator norm.
\end{rmk}

By means of the operators $\mathcal{Q}_{\gamma,i,\overline{p}}$ we define
$$\mathcal{Q}_{B}=\sum_{i}\sum_{\overline{p}\in E_{i,\vartheta}}\mathcal{Q}_{\gamma,i,\overline{p}},$$
so that, for a given function $\ell$ in $x^\gamma C^{\alpha}_{\Phi}(M\times[0,T])$, one has
\begin{equation}\label{DefnBdryParam}
\mathcal{Q}_B\ell=\sum_{i}\sum_{\overline{p}\in E_{i,\vartheta}}\psi_{i,\overline{p}}\mathbf{H}_{\gamma,\overline{p}}(\varphi_{i,\overline{p}}\ell).
\end{equation}
\begin{prop}\label{BoundariParametrixProp}
For every $0<\delta<1$ there exist $\varepsilon$ and $T$ positive and sufficiently small so that 
\begin{equation}\label{BoundaryParametrixMappingProp}
\begin{split}
&\mathcal{Q}_B:x^\gamma C^{\alpha}_\Phi(M\times[0,T])\rightarrow x^\gamma C^{2,\alpha}_\Phi(M\times[0,T]),\\
&\mathcal{Q}_B:x^\gamma C^{\alpha}_\Phi(M\times[0,T])\rightarrow x^\gamma \sqrt{t}C^{1,\alpha}_\Phi(M\times[0,T])
\end{split}
\end{equation}
are bounded operators.
Moreover, in terms of the function $\phi$ defined in \eqref{phisum} one has, for every $\ell\in x^\gamma C^{\alpha}_{\Phi}(M\times[0,T])$,
$$(\partial_t+a\Delta)(\mathcal{Q}_B\ell)=\phi \ell +R^1\ell+R^2\ell$$
with $\|R^1\|_{\op}\le\delta$ and $\|R^2\|_{\op}$ converging to $0$ as $T$ goes to $0$.
\end{prop}
\begin{proof}
The mapping properties in \eqref{BoundaryParametrixMappingProp} are a straightforward consequence of the mapping properties of the heat-kernel operator $\mathbf{H}$ (cf. Theorem \ref{MappingPropertiesTHM}) and by noticing that multiplication by either $\psi_{i,\overline{p}}$ or $\varphi_{i,\overline{p}}$ are bounded operators, thus preserving the regularity.
\medskip

For the second part of the statement we begin by explicitly computing $(\partial_t+a\Delta)(\mathcal{Q}_{B}\ell)$.
Since the sum defining $\mathcal{Q}_B$ in \eqref{DefnBdryParam} is locally finite, by Lemma \ref{LocalSolvabilityBoundaryLemma} we conclude
\begin{align*}
(\partial_t+a\Delta)\mathcal{Q}_{B}\ell=&\sum_{i}\sum_{\overline{p}\in E_{i,\vartheta}}(\partial_t+a\Delta)\left(\psi_{i,\overline{p}}\mathbf{H}_{\gamma,\overline{p}}(\varphi_{i,\overline{p}}\ell)\right)\\
=&\phi\ell+\sum_{i}\sum_{\overline{p}\in E_{i,\vartheta}}R^1_{i,\overline{p}}\ell+\sum_{i}\sum_{\overline{p}\in E_{i,\vartheta}}R^2_{i,\overline{p}}\ell.
\end{align*}
For simplicity let us denote $R^j\ell=\sum_i\sum_{\overline{p}\in E_{i,\vartheta}}R^j_{i,\overline{p}}\ell$ for $j=1,2$.
Lemma \ref{LocalSolvabilityBoundaryLemma} gives 
$$\|R^1_{i,\bar{p}}\ell\|_{\alpha}\le C\left(T^{\alpha/2}+\varepsilon^{\beta}\right)\varepsilon^{-\alpha}\|\ell\|_{\alpha}.$$
Hence, by letting $\|\ell\|_{\alpha}\le 1$ we find that the operator norm of $R^1$ is bounded by
$$\|R^1\|_{\op}\le C\left(T^{\alpha/2}+\varepsilon^{\beta}\right)\varepsilon^{-\alpha}.$$
Again, the $C$'s denote different uniform constants.
For a given $0<\delta<1$ and $C$ as in the above estimate, it is possible to choose $0<T<1$ and $\varepsilon<\min\{1,1/\overline{C}\}$ sufficiently small so that 
$$T^{\alpha/2}\varepsilon^{-\alpha}+\varepsilon^{\beta-\alpha}<\frac{\delta}{C};$$
and $x=\varepsilon$ is a smooth hypersurface.
This might be accomplished, for instance, by choosing 
$$T^{\alpha/2}< \frac{\delta}{2C}\varepsilon^\alpha;\;\;\varepsilon^{\beta}<\frac{\delta}{2C}\varepsilon^{\alpha}.$$ 
In concerns of the operator norm of $R^2$, the estimate follows directly by employing Lemma \ref{LocalSolvabilityBoundaryLemma}.
\end{proof}

\subsection{Construction of the Parametrix}\label{InteriorParametrixSec}

In \S \ref{BoundaryParametrixSec2} we constructed an approximate boundary parametric for an heat-type operator $P$.
Here, we will first construct an approximate parametrix $\mathcal{Q}_{I}$ for $P$ in the interior $M$ of $\overline{M}$.
After obtaining $\mathcal{Q}_{I}$, we will see that a combination of $\mathcal{Q}_{B}$, as in \eqref{DefnBdryParam}, and $\mathcal{Q}_{I}$, defined below in \eqref{InteriorParametrixDef}, will lead to an approximate parametrix $\mathcal{Q}$ for $P$ on the whole $\overline{M}$.
As it is usual in Operator Theory, we will then get rid of the error, arising from $\mathcal{Q}$ being an approximate parametrix, via von Neumann series resulting in the claimed parametrix $\mathbf{Q}$ for $P$.
\medskip

Let $0<\delta<1$ be fixed and consider $\varepsilon$ and $T$ as in Proposition \ref{BoundariParametrixProp}.
From $\varepsilon$ being fixed, it follows that an $\varepsilon$-neighborhood of $\partial\overline{M}$ is also fixed and the function $\phi$ (defined in \eqref{phisum}) is identically $1$ on this neighborhood.
The idea now is to cut off a neighborhood of $\partial\overline{M}$ from $\overline{M}$.
Let $M_\varepsilon:=\{p\in \overline{M}\,|\,x(p)\ge \varepsilon/2\}$.
Clearly $M_\varepsilon$ is a compact manifold with boundary, meaning that we can consider its double space $\widehat{M}$.
Recall that the double space consists of two copies of $M_\varepsilon$ glued along the boundary and, for compact manifolds with boundary, it is a compact manifold without boundary. 
Note that the double space construction does not lead to a smooth metric on $\widehat{M}$. 
In order to smooth it up we consider a smoothing of such a metric so that the metric on $\widehat{M}$ and the one on $M$ coincide on $M_{2\varepsilon}$. 
Moreover, in dealing with $\widehat{M}$, we are working away from the boundary $\partial\overline{M}$ of $\overline{M}$.  Thus, the $\alpha$-H\"{o}lder spaces are exactly the classical ones.
\medskip

The function $(1-\phi)$ is defined on $M_\varepsilon$, but by setting it to be zero on the second copy of $M_\varepsilon$, we can extend it to a function, still denoted by $(1-\phi)$, on the double space $\widehat{M}$.
Hence $(1-\phi)$ defines, in particular, a smooth cut off function over $M_\varepsilon$ in $\widehat{M}$. 
Similarly, let $\overline{P}$ denote the uniform parabolic extension of $P|_{M_\varepsilon}$ to $\widehat{M}$.
From classical parabolic PDE theory, it is well known that there exists a parametrix $\overline{Q}_I$ for the heat operator $\overline{P}$ so that the maps
\begin{equation}\label{usual-holder-map}
\begin{split}
&\widehat{Q}_I:C^{k,\alpha}({\widehat{M}}\times [0,T])\rightarrow C^{k+2,\alpha}({\widehat{M}}\times[0,T]),\\
&\widehat{Q}_I:C^{k,\alpha}({\widehat{M}}\times [0,T])\rightarrow \sqrt{t}C^{k+1,\alpha}({\widehat{M}}\times[0,T]),
\end{split}
\end{equation}
are bounded.
The idea is to use such a parametrix $\widehat{Q}_I$ and the boundary parametrix constructed above to construct a parametrix $\mathcal{Q}$ for the Cauchy problem \eqref{CPE1}.
\medskip

Note that, for a given function $\widehat{u}\in C^{k,\alpha}(\widehat{M}\times[0,T])$, the second mapping property in \eqref{usual-holder-map} implies $\widehat{Q}_I\widehat{u}\in\sqrt{t}C^{k+1,\alpha}(\widehat{M}\times[0,T])$. 
In order to turn $\widehat{Q}_I\widehat{u}$ into a function in $C^{k,\alpha}_\Phi(M\times[0,T])$, let us consider a cut off function $\widehat{\Psi}$ on $\widehat{M}$ so that $\widehat{\Psi}=1$ on $\supp(1-\phi)$.
We can now define the operator 
\begin{equation}\label{InteriorParametrixDef}
\mathcal{Q}_I:=M(\widehat{\Psi})\circ \widehat{Q}_I \circ M(1-\phi),
\end{equation}
As pointed out in the proof of Proposition \ref{BoundariParametrixProp}, multiplication by $\widehat{\Psi}$ and $(1-\phi)$ preserve the regularity and are bounded operators.
Therefore the operator $\mathcal{Q}_I$
\begin{align*}
\mathcal{Q}_I&:x^\gamma C^{k,\alpha}_\Phi(M\times[0,T])\xrightarrow{{M(1-\phi)}}C^{k,\alpha}(\widehat{M}\times[0,T])\xrightarrow{\widehat{Q}_I}\\
&\xrightarrow{\widehat{Q}_I}\sqrt{t}C^{k+1,\alpha}(\widehat{M}\times[0,T])\xrightarrow{M(\overline{\Psi})}\sqrt{t}C^{k+1,\alpha}(M_\varepsilon \times[0,T])
\end{align*} 
acts continuously. 
Moreover, since we are working away from the boundary of $\overline{M}$, the spaces $C^{k+1,\alpha}(M_\varepsilon\times[0,T])$ can be identified with the space $x^\gamma C^{k+1,\alpha}_\Phi(M_\varepsilon\times[0,T])$. We can hence conclude that the operator $\mathcal{Q}_I$ mapping
$$\mathcal{Q}_I:x^\gamma C^{k,\alpha}_\Phi(M\times[0,T])\rightarrow x^\gamma\sqrt{t}C^{k+1,\alpha}_\Phi(M\times[0,T])$$
is bounded. 
We can therefore construct an approximate parametrix $\mathcal{Q}$ for the operator $P$ by setting
$$\mathcal{Q}\ell=\mathcal{Q}_B \ell+\mathcal{Q}_I \ell.$$
In particular, in view of the construction above and Proposition \ref{BoundariParametrixProp}, one sees that
\begin{equation}
\begin{split}
\mathcal{Q}&:x^\gamma C^{\alpha}_\Phi(M\times[0,T])\rightarrow x^\gamma C^{2,\alpha}_\Phi(M\times[0,T])\\
\mathcal{Q}&:x^\gamma C^{\alpha}_\Phi(M\times[0,T])\rightarrow x^\gamma\sqrt{t}C^{1,\alpha}_\Phi(M\times[0,T])
\end{split}
\end{equation}
are bounded.
\begin{prop}\label{RightInverseProp}
Let $0<\alpha<\beta\le 1$ and consider $a \in C^{k,\beta}_{\Phi}(M\times [0,T])$ to be positive and bounded from below away from zero.
There exists $T_0>0$ sufficiently small so that the operator $\mathbf{Q}$ acts continuously when mapping
\begin{align*}
&\mathbf{Q}:x^\gamma C^{\alpha}_\Phi(M\times[0,T_0])\rightarrow x^\gamma C^{2,\alpha}_\Phi(M\times[0,T_0]),\\
&\mathbf{Q}:x^\gamma C^{\alpha}_\Phi(M\times[0,T_0])\rightarrow x^\gamma\sqrt{t}C^{1,\alpha}_\Phi(M\times[0,T_0]).
\end{align*} 
Moreover, for every function $\ell$ in $x^\gamma C^\alpha_\Phi (M\times [0,T])$, $\mathbf{Q}\ell$ is a solution of the inhomogeneous Cauchy problem
\begin{equation}
(\partial_t+a\Delta)u=\ell, \;\; u|_{t=0}=0.
\end{equation}
\end{prop}
\begin{proof}
Let $\ell$ be a function in $x^\gamma C^{\alpha}_\Phi(M\times[0,T])$. 
By Proposition \ref{BoundariParametrixProp} and the construction above one computes
$$(\partial_t+a\Delta)(\mathcal{Q}\ell)=\phi \ell+R^1\ell+R^2\ell+(1-\phi)\ell+R^3\ell;$$
where $R^1$ and $R^2$ are the ones arising from Proposition \ref{BoundariParametrixProp} while $R^3$ is given by
$$R^3\ell=[a\Delta,\overline{\psi}]\left(\overline{Q}_I\left((1-\phi)\ell\right)\right).$$
Clearly $R^3:x^\gamma C^{\alpha}_\Phi(M\times[0,T])\rightarrow x^\gamma C^{\alpha}_\Phi(M\times[0,T])$ is bounded.
Furthermore, the operator norm of $R^3$ can be estimated in the same way as it has been done for $R^2$ in Lemma \ref{LocalSolvabilityBoundaryLemma}. 
In particular, it follows that both $\|R^2\|_{\op}$ and $\|R^3\|_{\op}$  converge to $0$ as $T$ goes to $0$, while $\|R^1\|_{\op}<\delta$. 
We can now find $T_0$ sufficiently small so that, for every $t\le\min{T_0,T}$, by denoting $R:=R^1+R^2+R^3$, 
$$\|R\|_{\op}\le\|R^1\|_{\op}+\|R^2\|_{\op}+\|R^3\|_{\op}< 1.$$
It is now clear that $\Id+R$ is invertible, with inverse obtained via Neumann series of $R$. 
The claimed right parametrix of $P$ will then be 
$$\mathbf{Q}=\mathcal{Q}(\Id+R)^{-1}.$$
\end{proof}
\begin{rmk}
In the above statement, $T_0$ arises from $\|R^2\|_{\op}$ and $\|R^3\|_{\op}$ converging to $0$ for $T\rightarrow 0^+$.
So, since $\|R^1\|_{\op}\le \delta$ we can fix $1-\delta$ and find $T_0$ so that $\|R^2\|_{\op}+\|R^3\|_{\op}<1-\delta$ for every $t\le T_0$.
\end{rmk}
\begin{cor} \label{Emap}
Let $a\in C^{\beta}_{\Phi}(M\times [0,T])$ be positive and bounded from below away from zero.
Then there exists $T_0$ sufficiently small (depending on $(\beta - \alpha)$), and a bounded operator
$$\mathbf{E}:x^\gamma C^{2,\alpha}_\Phi(M)\rightarrow x^\gamma C^{2,\alpha}_\Phi(M\times [0,T_0]),$$
so that, for every $u_0$ in $x^\gamma C^{2,\alpha}_\Phi(M)$, $u=\mathbf{E}u_0$ is a solution of the homogeneous Cauchy problem
\begin{equation}
(\partial_t+a\Delta)u=0,\;\;
u\big|_{t=0}=u_0.
\end{equation}
\end{cor}
\begin{proof}
Since $u_0\in C^{2,\alpha}_\Phi(M)$, $a\Delta u_0$ lies in $C^{\alpha}_\Phi(M\times[0,T])$. 
Using the right parametrix for the inhomogeneous Cauchy problem constructed in Proposition \ref{RightInverseProp}, set 
$$\mathbf{E}u_0=u_0-\mathbf{Q}(a\Delta u_0).$$
An easy computation shows that $\mathbf{E}u_0$ solves the homogeneous Cauchy problem.
\end{proof}

Note that, unlike the statement of Theorem \ref{theorem2}, the last two results gives us a solution only on an interval $[0,T_{0}]$ which is possibly different from the initial interval $[0,T]$.  

\begin{proof}[\textbf{Proof of Theorem \ref{theorem2}}]
Consider a function $\ell \in x^{\gamma}C^{\alpha}_{\Phi}(M\times [0,T])$ and the Cauchy problem 
	\begin{equation}\label{abc}
	(\partial_{t}+a\Delta)u = \ell; \hspace{2mm} u|_{t=0} = 0,
	\end{equation}
From Proposition \ref{RightInverseProp}, we know that the Cauchy problem above admits a solution $u$ lying in $x^{\gamma}C^{2,\alpha}_{\Phi}(M\times [0,T_{0}])$.  
Clearly, if $T_0\ge T$ then the statement is true and there is nothing to prove.
Suppose, otherwise, that $T_0<T$.
We claim that the solution $u$ can be extend past $T_0$ meaning that we can find a $C^{2,\alpha}_{\Phi} (M\times[0,T])$ solution to \eqref{abc} which agrees with $u$ up to time $T_0$; therefore allowing us to find solutions defined on the whole time interval definition of the function $\ell$.
Let $\lambda\in (0,T_0)$ and consider the Cauchy problem
\begin{equation}\label{abc1}
(\partial_{t} + a\Delta)v_{1} = 0; \hspace{2mm} v_{1}|_{t=0} = u|_{t=T_{0}-\lambda},
\end{equation}
that is the homogeneous Cauchy problem with initial condition $u|_{t=T_0-\lambda}$.
From Corollary \ref{Emap} we know that \eqref{abc1} admits a solution, say $v_1$,  $v_1\in C^{2,\alpha}_\Phi(M\times [0,T_0])$ ($T_0$ is independent on the initial condition).
By performing a change of coordinates, i.e. $t\mapsto t+T_0-\lambda$ we can consider the function $v_1\in C^{2,\alpha}_\Phi(M\times[T_0-\lambda,2T_0-\lambda])$.

Similarly we can consider the "shifted" problem for \eqref{abc}. 
That is 
\begin{equation}\label{abc2}
(\partial_t+a\Delta)u_1=\ell(\_,t+T_0-\lambda),\,\,u_1|_{T_0-\lambda}=0.
\end{equation}
Again, by Proposition \ref{RightInverseProp}, a solution to \eqref{abc2}  $u_1\in x^\gamma C^{2,\alpha}_\Phi(M\times[T_0,2T_0-\lambda])$ exists.

Denote by $w$ the function $u_1+v_1$. 
Since $P$ is a linear operator we see that $w\in x^\gamma C^{2,\alpha}(M\times[T_0,2T_0-\lambda])$ satisfies
\begin{equation}\label{abc3}
(\partial_t+a\Delta)w=\ell(\_,t+T_0-\lambda),\,\,w|_{t=0}=u|_{T_0-\lambda}.
\end{equation}
It is now the time to point out that the function $u$ satisfies \eqref{abc3} in $[T_0-\lambda,T_0]$ as well.
Therefore, from Corollary \ref{UniquenessOfAny} we conclude that $u(\_,t)=w(\_,t)$ for every $t\in [T_0-\lambda,T_0]$.
This means that we can $C^2$-glue $u$ and $w$ giving rise to
\begin{equation*}
	\widetilde{u}(p,t) = 
	\begin{cases}
	u(p,t)& \mbox{ if } \hspace{2mm} 0\le t \le T_{0} \\
	w(p,t)& \mbox{ if } \hspace{2mm}T_{0}<t\le 2T_{0}-\lambda.
\end{cases}	
	\end{equation*}
Now, if $2T_{0}-\lambda \ge T$, the result is proved.
If not, repeat the process with $\widetilde{u}$ until $nT_{0} - n\lambda \ge T$ (which is possible in a finite number of repetitions since $[0,T]$ is compact).  Thus we have an extension of $u$ defined on $M\times [0,T]$. 
\medskip

Note that this extension was obtained employing the parametrix construction, i.e. the maps $\mathbf{Q}$ and $\mathbf{E}$.
Such maps are bounded, thus the extended map $\mathbf{Q}$ so that $\ell\mapsto \widetilde{u}$ is also bounded.
The proof of Corollary \ref{Emap} implies that the operator $\mathbf{E}$ can be extended as well, thus completing the proof.

\end{proof}

\section{Generalization of short-time existence}\label{ShortTimeExistenceSec}

\textcolor{red}{}

In \S \ref{ParametrixSection} we proved the existence of solutions for non-homogeneous Cauchy problems with vanishing initial condition (cf. Theorem \ref{theorem2}).
In the analysis of geometric flows, as the Yamabe flow or the Mean Curvature flow, one deals with quasi-linear heat-type Cauchy problems.
It is therefore useful to introduce some non-linearity in the heat-type Cauchy problems in the setting of $\Phi$-manifolds.
\medskip

For $0<\alpha<\beta\le 1$ and $a\in C^{k,\beta}_{\Phi}(M\times[0,T])$ as in the assumptions of Theorem \ref{theorem2}.
We are interested in Cauchy problems of the form 
\begin{equation}\label{cauchyshort}
(\partial_{t} + a\Delta)u = F(u), \;\; u|_{t=0}=0,
\end{equation}
with the operator $F$ subject to some restrictions. 
We have already seen something like this, namely Theorem \ref{cor-paper1}; indeed under the assumption $a=1$ and $F$ satisfying
\begin{enumerate}
\item $F:x^{\gamma}C^{k+2,\alpha}_{\Phi}(M\times [0,T])\rightarrow C^{k,\alpha}_{\Phi}(M\times [0,T])$;
\item $F$ can be written as a sum $F = F_{1} + F_{2}$ with
\begin{itemize}
\item[(i)]$F_{1}:x^{\gamma}C^{k+2,\alpha}_{\Phi}\rightarrow x^{\gamma}C^{k+1,\alpha}_{\Phi}(M\times [0,T]),$
\item[(ii)]$F_{2}:x^{\gamma}C^{k+2,\alpha}_{\Phi}\rightarrow x^{\gamma}C^{k,\alpha}_{\Phi}(M\times [0,T]);$
\end{itemize}
\item For $u,u' \in x^{\gamma}C^{k+2,\alpha}_{\Phi}(M\times [0,T])$ with $\|\cdot\|_{k+2,\alpha,\gamma}$-norm bounded from above by some $\eta>0$, i.e. $\|u\|_{k+2,\alpha,\gamma},\|u'\|_{k+2,\alpha,\gamma} \le \eta$,
there exists some $C_{\eta}>0$ such that
\begin{itemize}
\item[(i)] $\|F_{1}(u) - F_{1}(u')\|_{k+1,\alpha,\gamma} \le C_{\eta}\|u-u'\|_{k+2,\alpha,\gamma}$, $\|F_{1}(u)\|_{k+1,\alpha,\gamma} \le C_{\eta}\|u\|_{k+2,\gamma,\alpha},$
\item[(ii)] $\|F_{2}(u) - F_{2}(u')\|_{k,\alpha,\gamma} \le C_{\eta}\max\{\|u\|_{k+2,\alpha,\gamma},\|u'\|_{k+2,\alpha,\gamma}\}\|u-u'\|_{k+2,\alpha,\gamma}$, \newline $\|F_{2}(u)\|_{k,\alpha,\gamma} \le C_{\eta}\|u\|^{2}_{k+2,\alpha,\gamma},$
\end{itemize}
\end{enumerate}
Theorem \ref{cor-paper1} guarantees existence and uniqueness of solution to the Cauchy problem aforementioned.  It should be noted, on the other hand, that the proof for such result (c.f. \cite[pg. 30-31]{paper1}) uses only the mapping properties of the heat-kernel operator $\mathbf{H}$ that hold for the parametrix $\mathbf{Q}$.  
Therefore, one can naturally extend the result to the parametrix constructed in \S \ref{ParametrixSection}, providing a proof for our last main result that is Corollary \ref{theorem4}.
\begin{rmk}
Contrarily to the same statement for the nonlinear heat equation with constant coefficient, we can not provide higher regularity, that is a solution $u^*$ existing in $C^{k+2,\alpha}_{\Phi}(M\times [0,T'])$ for some $T'$ small enough.
This is fairly reasonable and it should attainable.
Unfortunately the estimates in the error term $R^1$ in Lemma \ref{LocalSolvabilityBoundaryLemma} do not seem to extend easily to higher regularity, due to some problems arising in the estimate of the sup-norm of the coefficient $a$ in case $a\in C^{k,\beta}_{\Phi}(M\times[0,T])$.
\end{rmk}

As mentioned at the beginning of this section the operator $F$ will allow us to deal with some non-linear heat-type Cauchy problems.
We want to conclude this work by explaining in a bit more details why this is the case.

A generic quasi-linear second order parabolic Cauchy problem on $M$ if of the form
\begin{equation}\label{quasilinparabolic}
\partial_t u=Lu,\;\;u|_{t=0}=u_0
\end{equation}
for some suitable function $u_0$ where $Lu=a^{ij}(p,t,u,\nabla u)D_{ij}u+b(p,t,u,\nabla u)$ with $D_{ij}$ being a second order partial differential operator. 
(Note that in order to have parabolicity, one needs that the Frech\'{e}t derivative of $L$ is indeed an elliptic operator with eigenvalues bounded away from zero).
In order to conclude short time existence of solutions to \eqref{quasilinparabolic} one usually argues by means of perturbations; that is, if we stay "close" to the the initial condition $u_0$ we may find some evolution of $u_0$ in terms of the equation in \eqref{quasilinparabolic} for short time.
This is equivalent to consider $u=u_0+v$ and derive an equation for $v$ from $\partial_t u=Lu$.
This will lead to a new Cauchy problem of the form 
\begin{equation}\label{LinearizedCauchy}
\partial_t v=L_0 v,\;\;v|_{t=0}=0.
\end{equation}
Now the operator $L_0$ is some sort of linearization of the operator $L$.
As one can expect, the operator $L_0$ might not be of the form $L_0=a\Delta-F$ with $F$ satisfying the conditions $(1),(2)$ and $(3)$ in the hypothesis of Corollary \ref{theorem4}.
That really depends on the quasi-linear operator $L$ at hand.
Therefore a unique treatment for every quasi-linear parabolic operators is impossible.
Finally, we want to point out that a linearization of the form $a\Delta+F$ with $F$ satisfying the three condition in Corollary \ref{theorem4} is expect for most of the geometric flows.
Indeed in such a case one deals with quasi-linear evolution operators containing, as higher order derivative term, a "time-dependent" Laplacian (see e.g. Mean Curvature flow) or some power of $u$ multiplying a (fixed-in time) Laplacian (e.g. Yamabe flow).

\bibliographystyle{amsalpha-lmp}

\end{document}